\documentclass[12pt]{amsart}
\usepackage{latexsym,amsthm,amsmath,amssymb,amsfonts}
\usepackage{tikz}\usetikzlibrary{calc,matrix,arrows}

\theoremstyle{plain}
\newtheorem{thm}{Theorem}[section]
\newtheorem{main}{Theorem}\renewcommand{\themain}{\Alph{main}}
\newtheorem{lem}[thm]{Lemma}
\newtheorem{prop}[thm]{Proposition}
\newtheorem{cor}[thm]{Corollary}
\theoremstyle{definition}
\newtheorem{defn}[thm]{Definition}
\newtheorem{rem}[thm]{Remark}
\newtheorem{exmp}[thm]{Example}

\newcommand{\R}{\mathbb{R}}
\newcommand{\F}{\mathbb{F}}
\newcommand{\sph}{\mathbb{S}}

\newcommand{\trees}{\mathsf{Trees}}\newcommand{\htrees}{\mathsf{HT}}
\newcommand{\nctrees}{\mathsf{NCTrees}}\newcommand{\nchtrees}{\mathsf{NCHT}}
\newcommand{\forests}{\mathsf{Forests}}\newcommand{\hforests}{\mathsf{HF}}

\newcommand{\parts}{\mathsf{Part}}\newcommand{\ncparts}{\mathsf{NCPart}}

\newcommand{\apartment}{\mathsf{Apart}}\newcommand{\sphere}{\mathsf{Sphere}}
\newcommand{\chamber}{\mathsf{Chamber}}\newcommand{\simplex}{\mathsf{Simplex}}
\newcommand{\poset}{\mathsf{Poset}}
\newcommand{\bool}{\mathsf{Bool}}
\newcommand{\dih}{\mathsf{Dih}}
\newcommand{\complex}{\mathsf{Complex}}

\newcommand{\link}{\mathsf{Link}}
\newcommand{\refl}{\mathsf{Refl}}
\newcommand{\rev}{\mathsf{rev}}
\newcommand{\linear}{\mathsf{Linear}}

\newcommand{\sym}{\mathsf{Sym}}

\newcommand{\catalan}{\mathsf{Cat}}

\newcommand{\size}{\mathrm{size}} 

\newcommand{\reuler}{\widetilde \chi}

\newcommand{\card}[1]{|#1|}

\newcommand{\poly}[2]{(0:#2) \foreach \x in {1,2,...,#1} {--({\x*360/#1}:#2)}--cycle;}

\tikzstyle{YellowPoly}=[thin,color=black,fill=yellow!20,join=bevel]
\tikzstyle{GreenPoly}=[thick,color=green!50!black,fill=green!30,join=bevel]
\tikzstyle{BlueLine}=[very thick,color=blue,join=bevel]
\tikzstyle{RedLine}=[very thick,color=red,join=bevel]
\tikzstyle{GreenLine}=[very thick,color=green!50!black,join=bevel]
\tikzstyle{BluePoly}=[BlueLine,fill=blue!20]
\tikzstyle{RedPoly}=[RedLine,fill=red!20]
\tikzstyle{heposet}=[shape=circle,draw,color=blue,fill=yellow!20,join=bevel]
\tikzstyle{smallDot}=[draw,shape=circle,color=black,fill=black,inner sep=.75pt]
\tikzstyle{Rect}=[fill=yellow!20,rounded corners,minimum width=.9cm,minimum height=.9cm,draw]
\newcommand{\drawAngle}[3]{\path #1 -- ++#2 coordinate(temp0) ++#3
  coordinate(temp1) #1 -- ++#3 coordinate(temp2); \draw (temp0)--(temp1)--(temp2);}
\newcommand{\drawAngleD}[3]{\path #1 -- ++#2 coordinate(temp0) ++#3
  coordinate(temp1) #1 -- ++#3 coordinate(temp2); \draw[dashed] (temp0)--(temp1)--(temp2);}

\begin{document}

\title[Noncrossing hypertrees]{Noncrossing hypertrees}

\author{Jon McCammond}

\date{\today}

\begin{abstract}
  Hypertrees and noncrossing trees are well-established objects in the
  combinatorics literature, but the hybrid notion of a noncrossing
  hypertree has received less attention.  In this article I
  investigate the poset of noncrossing hypertrees as an induced
  subposet of the hypertree poset.  Its dual is the face poset of a
  simplicial complex, one that can be identified with a generalized
  cluster complex of type $A$.  The first main result is that this
  \emph{noncrossing hypertree complex} is homeomorphic to a piecewise
  spherical complex associated with the noncrossing partition lattice
  and thus it has a natural metric.  The fact that the order complex
  of the noncrossing partition lattice with its bounding elements
  removed is homeomorphic to a generalized cluster complex was not
  previously known or conjectured.

  The metric noncrossing hypertree complex is a union of unit spheres
  with a number of remarkable properties: 1) the metric subspheres and
  simplices in each dimension are both bijectively labeled by the set
  of noncrossing hypertrees with a fixed number of hyperedges, 2) the
  number of spheres containing the simplex labeled by the noncrossing
  tree $\tau$ is the same as the number simplices in the sphere
  labeled by the noncrossing tree $\tau$, and 3) among the maximal
  spherical subcomplexes one finds every normal fan of a metric
  realization of the simple associahedron associated to the cluster
  algebra of type $A$.  In particular, the poset of noncrossing
  hypertrees and its metric simplicial complex provide a new
  perspective on familiar combinatorial objects and a common context
  in which to view the known bijections between noncrossing partitions
  and the vertices/facets of simple/simplicial associahedra.
\end{abstract}

\keywords{Noncrossing trees, hypertrees, noncrossing partitions,
  associahedra, cluster algebras}

\maketitle

\section*{Introduction}

The properties of hypertrees are well-documented \cite{MM96, Wa98,
  Ka99, BMMM01, MM04, JMM06, JMM07, Ch07, Og13} as are the properties
of noncrossing trees \cite{No98, FlNo99, DeNo02, PaPr02, Pa03, Ho03,
  ChYa06, SuWa09, LvSa14}, but the hybrid concept of a noncrossing
hypertree has received much less attention.  The poset of noncrossing
hypertrees, viewed as an induced subposet of the better-known
hypertree poset has upper intervals that are Boolean lattices and thus
is (the dual of) the face lattice of a simplicial complex that I call
the \emph{noncrossing hypertree complex}.\footnote{There is an
  alternative encoding of noncrossing hypertrees as decompositions of
  even-sided polygons into even-sided subpolygons and as such the
  noncrossing hypertree complex can be identified with the generalized
  cluster complex of type $A$ with $m = 2$.  See
  \S\ref{sec:ncht-complex} for details. In this guise it has been
  studied before but different aspects of its structure are visible
  when its simplex labels are viewed as noncrossing hypertrees. In
  particular the homeomorphism established in
  Theorem~\ref{main:geometry} is new and its proof relies heavily on
  the structure of the noncrossing hypertrees.}

The first main result is to identify the topology of the noncrossing
hypertree complex as that of the piecewise spherical metric simplicial
complex that is the link of the long diagonal edge in the orthoscheme
complex of the noncrossing partition lattice, which, for brevity, I
refer to as the \emph{noncrossing partition link}.  Through this
connection the noncrossing hypertree complex can be turned into a
geometric object.

\begin{main}[Topology and Geometry]\label{main:geometry}
  The noncrossing hypertree complex is naturally homeomorphic to the
  noncrossing partition link. As a consequence, the piecewise
  spherical metric on the latter induces a piecewise spherical metric
  on the former.
\end{main}

One way to view the noncrossing hypertree complex is as a simpler
simplicial structure on the noncrossing partition link.  This
simplified metric structure has a number of remarkable properties.
For example, there is a duality between its simplices and its metric
subspheres.

\begin{main}[Simplices and Spheres]\label{main:duality}
  Every spherical simplex in the metric noncrossing hypertree complex
  of any dimension is contained in a subcomplex isometric to a unit
  sphere of the same dimension.  In fact, in the top dimension (and
  conjecturally in all dimensions) there is a natural map from
  simplices to spherical subcomplexes that establishes a bijection
  between these two sets.
\end{main}

The spheres referred to in the theorem are called \emph{special
  spheres}.  Because the noncrossing partition link is a subcomplex of
a spherical building, its top-dimensional simplices and
top-dimensional spherical subcomplexes are sometimes called chambers
and apartments, respectively.  In the new simplicial structure of the
noncrossing hypertree complex the top-dimensional spheres are the same
as before so I still call them \emph{apartments} and I call the
top-dimensional simplices in the noncrossing hypertree complex
\emph{tree chambers}.  They are amalgamations of the original chambers
in the noncrossing partition link which, for the sake of clarity, I
call \emph{partition chambers}.  In this language
Theorem~\ref{main:duality} states that the set of tree chambers and
the set of apartments in the noncrossing hypertree complex are both
bijectively labeled by noncrossing trees.  The next result establishes
some additional aspects of this duality.

\begin{main}[Bijections]\label{main:bijections}
  For every noncrossing hypertree $\tau$ there is a bijection between
  the number of special spheres containing the tree simplex labeled
  $\tau$ as a top-dimensional simplex and the set of tree simplices
  contained in the special sphere labeled $\tau$.  When $\tau$ is a
  noncrossing tree, this means that there is a bijection between $\{
  \sigma \mid \chamber(\tau) \in \apartment(\sigma)\}$, the set of
  apartments containing the tree chamber labeled $\tau$ and $\{ \sigma
  \mid \chamber(\sigma) \in \apartment(\tau)\}$, the set of tree
  chambers in the apartment labeled~$\tau$.
\end{main}

As $\tau$ varies, these numbers vary as well, from a minimum that is a
power of $2$ to a maximum, conjecturally, that is a Catalan
number.\footnote{The Catalan numbers are the maximum values in all the
  cases where computer investigation is feasible, but I do not
  currently have a proof that this is the maximal value.}  The
structure of these extremal simplicial spheres is easy to describe.
The minimum value corresponds to an orthoplex, the generalization of
the octahedron also known as a cross-polytope, and the maximum value
corresponds to a simplicial associahedron.

\begin{main}[Associahedra]\label{main:associahedra}
  Let $\tau$ be a noncrossing tree.  If the tree chamber labeled
  $\tau$ consists of a single partition chamber then the apartment
  labeled $\tau$ is a simplicial associahedron.  In addition, the
  variety of simplicial associahedra produced in this way include all
  of the simplicial associahedra that are normal fans to the type $A$
  simple associahedra constructed by Hohlweg and Lange.
\end{main}

The simplicial associahedra that this theorem produces are closely
related to Reading's $c$-Cambrian fans \cite{Re06,ReSp09}.

\subsection*{Structure of the article}

The first three sections establish basic properties of the noncrossing
hypertree complex and the next two sections establish basic properties
of the noncrossing partition link.  The proof of
Theorem~\ref{main:geometry} is spread over the three sections after
that.  Once these foundations are in place, the remaining main
theorems are proved in the final sections.

\section{Hyperforests and Hypertrees\label{sec:htrees}}

A simple graph is a set of vertices and a collection of $2$-element
subsets called edges, a forest is simple graph with no cycles, and a
tree is a connected forest.  Hypergraphs, hyperforests and hypertrees
are expanded versions of these notions.

\begin{defn}[Hypergraphs]
  A \emph{hypergraph} is a collection of vertices, usually identified
  with the first few natural numbers, and a collection of subsets of
  the vertices called \emph{hyperedges} where each hyperedge must
  contain at least $2$ elements.  Familiar graph definitions are
  slightly modified to accommodate this change.  A \emph{path of
    length $k$} in a hypergraph is an alternating sequence
  $(v_0,e_1,v_1,\ldots,e_k,v_k)$ of vertices $v_i$ and hyperedges
  $e_i$, starting and ending with a vertex, where each hyperedge $e_i$
  contains the vertices $v_{i-1}$ and $v_i$.  Its \emph{endpoints} are
  $v_0$ and $v_k$, its length is $k$ and paths of positive length are
  \emph{nontrivial}.  A path is \emph{simple} if all of its edges and
  vertices are distinct, a \emph{cycle} if its endpoints are equal,
  and a \emph{simple cycle} if all of its edges are distinct and all
  of its vertices are distinct except that its endpoints are equal.  A
  hypergraph is \emph{connected} if every pair of vertices are the
  endpoints of a path, a \emph{hyperforest} if there are no nontrivial
  simple cycles and a \emph{hypertree} if it is a connected
  hyperforest.  A pair of hyperedges are said to be \emph{weakly
    disjoint} when they have at most one vertex in common and the
  hyperedges of a hyperforest are pairwise weakly disjoint because two
  hyperedges with two common vertices form a simple cycle.  A
  hypergraph whose hyperedges are pairwise disjoint can be viewed as a
  \emph{partition}. It has a block for every hyperedge and a singleton
  block for each vertex not contained in any hyperedge.  A hypertree
  is shown in Figure~\ref{fig:hypertree}.
\end{defn}

\begin{figure}
  \begin{tikzpicture}[scale=.7]
    \coordinate (1) at (1,2);
    \coordinate (2) at (2,1);
    \coordinate (3) at (1,1);
    \coordinate (4) at (1,0);
    \coordinate (5) at (0,0);
    \coordinate (6) at (0,-1);
    \coordinate (7) at (-1,0);
    \coordinate (8) at (-1,1);
    \coordinate (9) at (0,1);
    \draw[BlueLine] (1)--(3)--(2);
    \draw[BluePoly] (3)--(4)--(5)--cycle;
    \draw[BlueLine] (5)--(6);
    \draw[BluePoly] (5)--(7)--(8)--(9)--cycle;
    \foreach \x in {1,2,...,9} {\fill (\x) circle (1mm);}
    \draw (1) node[above] {\small $1$};
    \draw (2) node[right] {\small $2$};
    \draw (3) node[below right] {\small $3$};
    \draw (4) node[below right] {\small $4$};
    \draw (5) node[below right] {\small $5$};
    \draw (6) node[below] {\small $6$};
    \draw (7) node[below left] {\small $7$};
    \draw (8) node[above left] {\small $8$};
    \draw (9) node[above] {\small $9$};
  \end{tikzpicture}
  \caption{A hypertree on $9$ vertices with hyperedges $\{1,3\}$,
    $\{2,3\}$, $\{3,4,5\}$, $\{5,6\}$ and
    $\{5,7,8,9\}$.\label{fig:hypertree}}
\end{figure}
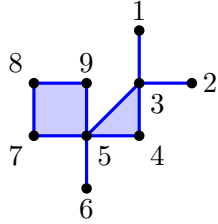

Although the focus of this article is noncrossing hypertrees,
hyperforests are useful when establishing results by induction.
References to hyperforests are rare (one of few is \cite{Kn05}) but
they are definitely worthy studying since the naturally defined
bounded graded hyperforest poset includes trees, hypertrees, forests
and partitions, as well as the noncrossing versions of these objects.

\begin{defn}[Hyperforest poset]
  The set of all hyperforests on a fixed number of vertices comes
  equipped with a natural partial order where $\sigma \leq \tau$ if
  and only if every hyperedge of $\sigma$ is a subset of a hyperedge
  of $\tau$.\footnote{This ordering agrees with the standard ordering
    on partitions, but it is the dual of the standard ordering on
    hypertrees found in the literature.}  The result is a bounded
  graded poset $\hforests_{n+1}$ where $n+1$ denotes the number of
  vertices.  The unique minimal element is the \emph{trivial
    hyperforest} $t$ with no hyperedges.  The unique maximal element
  has a single hyperedge containing all the vertices which, in
  anticipation of its later uses, I call the \emph{Coxeter
    hypertree}. There are two types of covering relations.  The
  \emph{first type} adds a new edge, i.e. a hyperedge with only $2$
  vertices, whose vertices belong to distinct connected components.
  The \emph{second type} involves replacing two hyperedges that share
  a vertex with the hyperedge that is their union.  This process is
  called \emph{merging} and its inverse is called \emph{splitting}.
  Note that a hyperedge needs to contain at least $3$ vertices in
  order to be splittable.  From this description of the covering
  relations, it is easy to show that the \emph{height} of a
  hyperforest $\tau$ is $2\card{V}- \card{E} - 2\card{C}$ where $V$ is
  the set of vertices, $E$ is the set of hyperedges and $C$ is the set
  of connected components.
 \end{defn}

Let $\tau$ be a hyperforest with $n+1$ vertices.  It has height $0$ if
$\tau$ is the trivial hyperforest $t$, $\card{E}$ if $\tau$ is a
forest, $n$ if $\tau$ is a tree, $2n - \card{E}$ if $\tau$ is a
hypertree and $2n-1$ if $\tau$ is the Coxeter hypertree.  The choice
of the variable $n$ to denote the number of vertices minus one, rather
than the number of vertices itself, is not an accident or a mistake.
It is a notational irritant with potential future benefits.

\begin{rem}[Rank]
  The symmetric group $\sym_{n+1}$ naturally acts on $\R^{n+1}$ by
  permuting coordinates and it also acts on the $n$-dimensional
  subspace of vectors where the sum of the coordinates is equal to
  $0$.  This action shows that the symmetric group $\sym_{n+1}$ is an
  example of a finite group of euclidean isometries generated by
  reflections, the reflections in this case being the transpositions
  $(i,j)$ that switch the $i^{\textrm{th}}$ and $j^{\textrm{th}}$
  coordinates pointwise fixing the hyperplane defined by the equation
  $x_i = x_j$.  More specifically, the symmetric group $\sym_{n+1}$ is
  the finite Coxeter group of type $A_n$, where the $n$ indicates,
  among other things, the size of the smallest reflection generating
  set, a number called its \emph{rank}.  All of the results presented
  here, conjecturally at least, are a ``type $A$'' version of a more
  general theory.  It is with these future developments in mind that
  the variable $n$ is usually used to denote the rank of the
  associated Coxeter group.  In a currently hypothetical general
  theory of hyperforests for arbitrary Coxeter groups, the set
  $\hforests_{n+1}$ of all hyperforests on $n+1$ vertices would be
  denoted $\hforests(A_n)$.
\end{rem}

\begin{figure}
  \begin{tikzpicture}[scale=.7]
    \def\r{4}
    \def\s{4}
    \draw (0,-\r) to [out=30,in=270] (2+\s,0) to [out=90,in=-30] (0,\r);
    \draw (0,-\r) to [out=150,in=270] (-2-\s,0) to [out=90,in=-150] (0,\r);
    \draw (0,-\r) to [out=60,in=270] (2-\s,0) to [out=90,in=-120] (0,\r);
    \draw (0,-\r) to [out=120,in=270] (-2+\s,0) to [out=90,in=-60] (0,\r);
    \draw (-\s,0) ellipse (2cm and 7mm);
    \draw (-\s,0) node {$\trees_{n+1}$};
    \draw (-2.75,.5*\r) node {$\htrees_{n+1}$};
    \draw (-2.75,-.5*\r) node {$\forests_{n+1}$};
    \draw (0,0) node {$\hforests_{n+1}$};
    \draw (\s,0) node {$\parts_{n+1}$};
    \fill[color=black] (0,\r) circle (1mm);
    \fill[color=black] (0,-\r) circle (1mm);
  \end{tikzpicture}
  \caption{The poset $\hforests_{n+1}$, its subposets
    $\forests_{n+1}$, $\parts_{n+1}$ and $\htrees_{n+1}$ and its
    subset $\trees_{n+1}$.\label{fig:hfp}}
\end{figure}
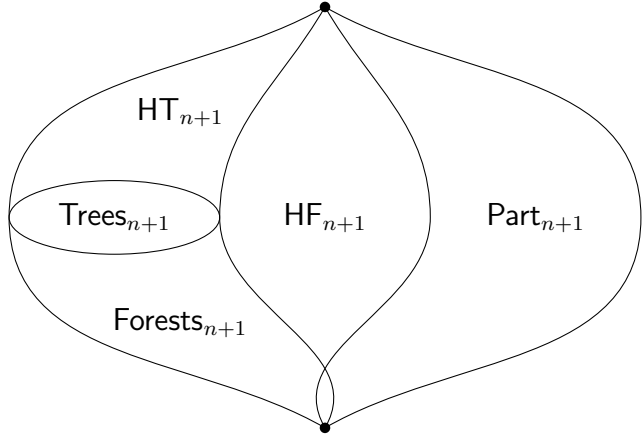

The hyperforest poset and its various substructures are shown in
Figure~\ref{fig:hfp}.  A collection of pairwise disjoint edges defines
both a forest and a partition, hence the partial overlap.

\begin{rem}[Forests, trees and hypertrees]
  The forests, trees and hypertrees inside the hyperforest poset can
  be identified by focusing on the two types of covering relations.
  Forests are connected to the trivial hyperforest by covering
  relations that add a new edge, hypertrees are connected to the
  Coxeter hypertree by covering relations that merge two hyperedges,
  and trees are on the boundary between these two subposets.
  Moreover, every maximal chain from the trivial hyperforest to the
  Coxeter hypertree contains exactly $n$ covering relations that add
  edges and $n-1$ covering relations that merge hyperedges. If the
  covering relations are labeled by the endpoints of the new edge in
  the first case and the pair of merged hyperedges in the second, then
  the steps along any maximal chain can be reordered so that all of
  the edge additions take place before any of the hyperedge mergers
  and this reordering does not change the set of labels on the
  covering edges.  The reordered maximal chain passes through a tree
  at height $n$ with forests below and hypertrees above.
\end{rem}

\begin{defn}[Hypertree poset]
  The set of all hypertrees on $n+1$ vertices form an induced subposet
  $\htrees_{n+1}$ of the hyperforest poset.  The grading and the
  covering relations in this \emph{hypertree poset} come from the
  hyperforest poset but all covering relations are of the second type.
  It is bounded above by the Coxeter hypertree and its minimal
  elements are trees.  Kalikow \cite{Ka99} and Warme \cite{Wa98}
  proved that its size is given by the formula $\card{\htrees_{n+1}} =
  􏰆\sum_k (n+1)^{k-1} S(n,k)$, where $S(n,k)$ is a Stirling number of
  the second kind.  For more information see the sequence
  \texttt{oeis:A030019} in the Online Encyclopedia of Integer
  Sequences \cite{oeis} and for its uses in geometric group theory,
  see \cite{MM96, BMMM01, MM04, JMM06, JMM07, Pi12} .
\end{defn}

The \emph{degree} of a vertex is the number of hyperedges containing
it and the total number of hyperedges is determined by its vertex
degrees.

\begin{prop}[Vertex degrees]\label{prop:degrees}
  If $\tau$ is a hypertree with vertices $V$ and hyperedges $E$, then
  $\sum_{v\in V} (\deg(v) -1) = \card{E} -1$. In other words, if one
  expects there to be one hyperedge and vertices to have degree $1$,
  then the sum of the excess vertex degrees is the number of excess
  hyperedges.
\end{prop}

\begin{proof}
  The equation holds for the Coxeter hypertree and it is preserved as
  hyperedges are split.
\end{proof}

The \emph{size} of a hyperedge is the number of vertices it contains.
For hypertrees the number of vertices is determined by its hyperedge
sizes.

\begin{prop}[Hyperedge sizes]\label{prop:edge-sizes}
  For a connected hypergraph $\tau$ with vertices $V$ and hyperedges
  $E$, $\sum_{e \in E} (\size(e)-1) = \card{V}-1$ if and only if
  $\tau$ is a hypertree.
\end{prop}

\begin{proof}
  In a connected hypergraph it is possible to order the hyperedges so
  that each hyperedge except the first has at least one vertex in the
  union of the hyperedges earlier in the list.  Thus for all connected
  hypergraphs $\sum_{e \in E} (\size(e)-1) \geq \card{V}-1$.  When
  $\tau$ is a hypertree each hyperedge in this ordering has exactly
  one vertex with the union of hyperedges earlier in the list and the
  inequality is an equality.  Conversely, suppose $\tau$ contains is a
  simple cycle. When the last hyperedge of the simple cycle is added,
  at least two of its vertices are not new and the inequality is
  strict.
\end{proof}


\begin{defn}[Partitions]\label{def:hf-part}
  Partitions form a bounded graded lattice, but it is important to
  note that its natural grading and covering relations are distinct
  from those of the hyperforest poset.  A covering relation in the
  partition lattice involves joining two blocks into a bigger block,
  but this changes $1$, $2$ or $3$ levels in hyperforest poset
  depending on whether neither, one or both of the blocks are more
  than a single element.  One way to identify the partitions inside
  the hyperforest poset is to restrict the partial order on
  $\hforests_{n+1}$ to the transitive closure of the second type of
  covering relation that merge two hyperedges sharing a vertex.  The
  connected components in the resulting order each have a unique
  maximum element that is a partition and all partitions are maximal
  in their connected component.  Thus the partitions index the
  connected components.  For each hyperforest $\tau$, the unique
  partition $\sigma$ above $\tau$ in this restricted ordering is
  obtained from $\tau$ by iteratively merging hyperedges with a vertex
  in common until no more mergers of this type are possible.  It is
  called the \emph{partition of $\tau$} and it is the partition
  determined by the connected components of $\tau$.  The map sending
  hyperforests to their partitions is an example of a closure
  operator.
\end{defn}

\section{Noncrossing Hypertrees\label{sec:ncht}}

This section introduces noncrossing versions of all of the
combinatorial structures defined in \S\ref{sec:htrees}.  The key
definition is that of noncrossing and weakly noncrossing subsets of
vertices of a convex polygon.

\begin{defn}[Noncrossing subsets]
  Two subsets of the vertices of a convex polygon in the plane are
  called \emph{noncrossing} when their convex hulls are completely
  disjoint and \emph{weakly noncrossing} when their convex hulls have
  at most one vertex, but no other points, in common.
\end{defn}

\begin{figure}
  \begin{tikzpicture}[scale=.7]
    \begin{scope}[xshift=-3cm]
      \foreach \x in {1,2,...,9} {\coordinate (\x) at (90-40*\x:1.5cm);}
      \filldraw[YellowPoly] (0,0) circle (2.5cm);
      \draw[BlueLine] (1)--(3)--(2) (4)--(1)--(5)--(6);
      \draw[BlueLine] (5)--(9)--(7) (9)--(8);
      \foreach \x in {1,2,...,9} {\fill (\x) circle (1mm);}
      \foreach \x in {1,2,...,9} {\draw (90-40*\x:2cm) node {\small
          \x};}
    \end{scope}
    \begin{scope}[xshift=3cm]
      \foreach \x in {1,2,...,9} {\coordinate (\x) at (90-40*\x:1.5cm);}
      \filldraw[YellowPoly] (0,0) circle (2.5cm);
      \draw[BlueLine] (1)--(3)--(2);
      \draw[BluePoly] (3)--(4)--(5)--cycle;
      \draw[BlueLine] (5)--(6);
      \draw[BluePoly] (5)--(7)--(8)--(9)--cycle;
      \foreach \x in {1,2,...,9} {\fill (\x) circle (1mm);}
      \foreach \x in {1,2,...,9} {\draw (90-40*\x:2cm) node {\small \x};}
    \end{scope}
    \end{tikzpicture}
    \caption{A noncrossing tree and a noncrossing hypertree.\label{fig:ncht}}
\end{figure}
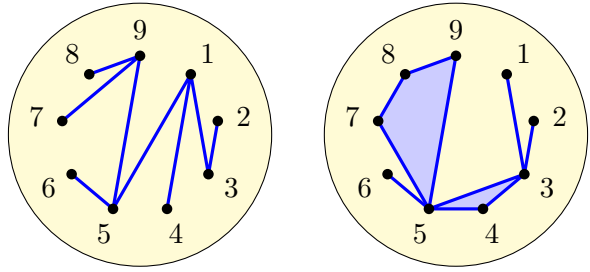

\begin{defn}[Noncrossing hypergraphs]\label{def:ncht}
  A hypergraph $\tau$ whose vertices have been identified with those
  of a convex polygon in the plane is called a \emph{noncrossing
    hypergraph} when its hyperedges, thought of as subsets of
  vertices, are pairwise weakly noncrossing.  A \emph{noncrossing
    tree, forest, hypertree, hyperforest}, or \emph{partition} is a
  noncrossing hypergraph that lives in this subcategory.  Note that
  because the hyperedges of a partition are completely disjoint, they
  are noncrossing in the strong sense and not merely weakly
  noncrossing.  When describing noncrossing hypergraphs, the vertices
  of the convex $k$-gon are usually labeled with the set $[k] :=
  \{1,2,\ldots,k\}$ where these integers are viewed as representing
  the integers modulo $k$.  They are placed in the order the vertices
  appear in the clockwise oriented boundary of the polygon.
\end{defn}

A noncrossing tree and noncrossing hypertree are shown in
Figure~\ref{fig:ncht}, and note that the noncrossing hypertree on the
right is the vertex-labeled hypertree of Figure~\ref{fig:hypertree}
redrawn inside a regular $9$-gon.  The noncrossing versions restrict
the size of the various sets and posets but their internal structure
remains much the same.  The relationships among the sets
$\trees_{n+1}$ and $\nctrees_{n+1}$ and the posets $\htrees_{n+1}$ and
$\nchtrees_{n+1}$ are shown in Figure~\ref{fig:nchtrees}.  A similar
diagram could be drawn for all of the substructures of the hyperforest
poset shown in Figure~\ref{fig:hfp}.

\begin{figure}
  \begin{tikzpicture}[scale=.7]
    \draw (0,0) ellipse (6cm and 1.5cm);
    \draw[line width=3pt, color=white] (-3,0) to [out=90,in=220] (0,6) to [in=90,out=320] (3,0);
    \draw (-3,0) to [out=90,in=220] (0,6) to [in=90,out=320] (3,0);
    \draw (-6,0) to [out=90,in=-160] (0,6) to [in=90,out=-20] (6,0);
    \draw (0,0) ellipse (3cm and 1cm);
    \draw (-4.5,0) node {$\trees_{n+1}$};
    \draw (0,0) node {$\nctrees_{n+1}$};
    \draw (-3.5,3) node {$\htrees_{n+1}$};
    \draw (0,3) node {$\nchtrees_{n+1}$};
    \fill[color=black] (0,6) circle (1mm);
  \end{tikzpicture}
  \caption{The poset $\htrees_{n+1}$, its subposet $\nchtrees_{n+1}$ and its
    subsets $\trees_{n+1}$ and $\nctrees_{n+1}$..\label{fig:nchtrees}}
\end{figure}
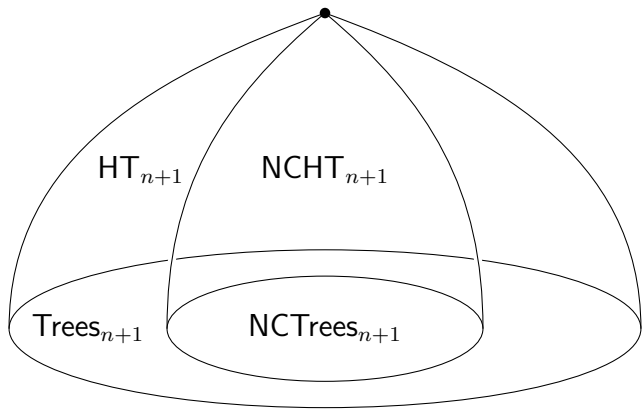

\begin{rem}[Noncrossing trees]
  The size of the set $\nctrees_{n+1}$ of all noncrossing trees on
  $n+1$ vertices is the generalized Catalan number $\frac{1}{2n+1}
  \binom{3n}{n}$, also known as the Fu\ss-Catalan number
  $\catalan^{(2)}$ of type $A_n$.  This was proved by Noy \cite{No98}
  and Panholzer and Prodinger \cite{PaPr02} among others.  For
  additional information, see the sequence \texttt{oeis:A001764} in
  \cite{oeis}.
\end{rem}

\begin{defn}[Noncrossing hypertree poset]\label{def:ncht-poset}
  The poset $\nchtrees_{n+1}$ of noncrossing hypertrees on $n+1$
  vertices is the induced subposet of $\htrees_{n+1}$ restricted to
  those hypertrees that are noncrossing.  The Coxeter hypertree
  remains the maximum element, the set $\nctrees_{n+1}$ of noncrossing
  trees are the minimal elements, the new covering relations are a
  subset of the old covering relations and the grading is as before.
\end{defn}

The covering relations in the noncrossing hypertree poset are
particularly important.  Since they correspond to covering relations
in the hypertree poset, they involve merging two hyperedges that share
a vertex, but there are restrictions.

\begin{defn}[Local linear orderings]\label{def:local-order}
  Let $v$ be a vertex of a noncrossing hyperforest $\tau$.  The
  $\deg(v)$ hyperedges of $\tau$ that contain $v$ have a \emph{local
    linear order} defined by standing at $v$ facing towards the
  interior of the polygon and linearly ordering these hyperedges as
  they occur from left to right.  Concretely, for distinct hyperedges
  $e$ and $e'$ containing $v$, $e < e'$ if and only if $e$ is to the
  left of $e'$ from the perspective of $v$ looking towards the
  interior of the polygon.  Note that when two hyperedges sharing a
  vertex in a noncrossing hyperforest are merged, the result remains a
  noncrossing hyperforest if and only if these hyperedges are adjacent
  in the local linear ordering of their shared vertex.
\end{defn}

These local linear orderings can be combined into a single poset
containing a lot of information.

\begin{defn}[Hyperedge posets]\label{def:edge-posets}
  Let $\tau$ be a noncrossing hyperforest.  The \emph{hyperedge poset
    of $\tau$} is a poset $\poset(\tau)$ whose elements are the
  hyperedges of $\tau$ and whose partial order is the transitive
  closure of the covering relations from local linear orderings of the
  hyperedges at each vertex of $\tau$.  Here are two examples to help
  clarify the definition.  The noncrossing hypertree shown on the
  right in Figure~\ref{fig:ncht} has an hyperedge poset that is
  linear.  In particular, its hyperedges are ordered as follows:
  $\{5,6\} < \{5,7,8,9\} < \{3,4,5\} < \{1,3\} < \{2,3\}$.  The
  noncrossing tree shown on the left has a more complicated hyperedge
  poset whose Hasse diagram is shown in Figure~\ref{fig:edge-poset}.
\end{defn}

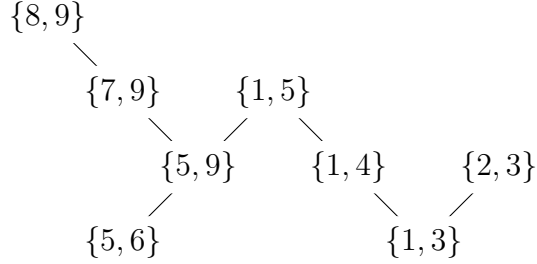
\begin{figure}
  \begin{tikzpicture}[scale=1]
    \draw (5,1)--(4,0)--(2,2)--(0,0) (1,1)--(-1,3);
    \begin{scope}
      \tikzstyle{every node}=[fill=white,rounded corners]
      \draw (0,0) node {$\{5,6\}$};
      \draw (1,1) node {$\{5,9\}$};
      \draw (2,2) node {$\{1,5\}$};
      \draw (3,1) node {$\{1,4\}$};
      \draw (4,0) node {$\{1,3\}$};
      \draw (5,1) node {$\{2,3\}$};
      \draw (0,2) node {$\{7,9\}$};
      \draw (-1,3) node {$\{8,9\}$};
    \end{scope}
  \end{tikzpicture}
\caption{The Hasse diagram for the hyperedge poset of the noncrossing
  tree shown on the left in Figure~\ref{fig:ncht}.  The hyperedges in
  the covering relations that lie in a straight line are locally
  linearly ordered by the vertex they share.\label{fig:edge-poset}}
\end{figure}

For an algebraic reason for this ordering, see Lemma~\ref{lem:po-nc}.
The examples illustrate key properties of hyperedge posets.

\begin{prop}[Hyperedge posets]\label{prop:edge-posets}
  For every noncrossing hypertree $\tau$, the Hasse diagram of its
  hyperedge poset $\poset(\tau)$, viewed as an undirected graph, is a
  tree. Similarly, the Hasse diagram of the hyperedge poset of a
  noncrossing hyperforest, viewed as an undirected graph, is a forest.
  In both cases, the edges of the Hasse diagram correspond to pairs of
  hyperedges that share a vertex and are adjacent in its local linear
  ordering.
\end{prop}

\begin{proof}
 First note that the fact that the noncrossing hypertree $\tau$ is
 connected means the Hasse diagram of $\poset(\tau)$ must be
 connected.  Next, from the way the partial order of $\poset(\tau)$ is
 generated, the only possible covering relations are those from the
 local vertex linear orders.  These are of the form described and
 there are $\deg(v)-1$ of these for each vertex $v$, but by
 Proposiion~\ref{prop:degrees}, the total number of possible covering
 relations is the minimum number needed to connect the Hasse diagram.
 Thus, all of the possible cover relations really are cover relations
 and there are none left over to make the undirected graph underlying
 the Hasse diagram anything other than a tree.  For noncrossing
 hyperforests consider one connected component at a time.
\end{proof}

The observation made in Definition~\ref{def:local-order} implies the
following.

\begin{prop}[Merging hyperedges]\label{prop:he-merge}
  Let $\tau$ be a noncrossing hyperforest and let $e$ and $e'$ be two
  hyperedges that share a vertex. The hyperforest $\tau'$ formed by
  merging $e$ and $e'$ remains noncrossing if and only if $e$ and $e'$
  are the endpoints of a covering relation in the hyperedge poset
  $\poset(\tau)$.  Moreover, when $\tau'$ remains noncrossing, the
  Hasse diagram of its hyperedge poset $\poset(\tau')$ is obtained
  from the Hasse diagram of $\poset(\tau)$ by shrinking the covering
  edge with endpoints $e$ and $e'$ and labeling the new element by the
  merged hyperedge.
\end{prop}

Iteratively applying Proposition~\ref{prop:he-merge} proves the
following.

\begin{prop}[Partition of a noncrossing hyperforest]\label{prop:part-nchf}
  The partition of a noncrossing hyperforest is a noncrossing
  partition.
\end{prop}

Proposition~\ref{prop:he-merge} also shows that the noncrossing
hypertrees above a fixed noncrossing hypertree $\tau$ correspond
exactly to those obtained by systematically collapsing a forest of
edges in the Hasse diagram of its hyperedge post while merging the
corresponding hyperedges.  That the result is independent of the order
in which they are collapsed and distinct for every distinct subset of
edges in the Hasse diagram should already be clear, but it will be
crystal clear once the collapsing/merging operation has been
reinterpreted as a polygon dissection.

\section{The Noncrossing Hypertree Complex\label{sec:ncht-complex}}

This section establishes a bijection between noncrossing hypertrees
and certain types of polygon dissections.  The noncrossing hypertrees
one step below the Coxeter hypertree are the key.

\begin{defn}[Basic noncrossing hypertrees]\label{def:basic}
  A hypertree with exactly two hyperedges is called a \emph{basic
    hypertree} and the basic noncrossing hypertrees are easy to
  describe.  A basic noncrossing hypertree $\tau$ is completely
  determined by the vertex its two hyperedges have in common and the
  edge on the boundary of the polygon that is not included in the
  convex hull of either hyperedge. The only restriction is that the
  shared vertex cannot be an endpoint of the missing boundary edge.
  Thus there are exactly $(n+1)(n-1) =n^2-1$ basic noncrossing
  hypertrees on $n+1$ vertices.  In each basic noncrossing hypertree,
  one of the two hyperedges is below the other in its hyperedge poset
  and it can be reconstructed from its lower hyperedge alone.  In
  Figure~\ref{fig:two-hyperedges} the $15$ basic noncrossing hypertrees on
  $5$ vertices are shown ordered according the ordering of their lower
  hyperedges under inclusion.
\end{defn}

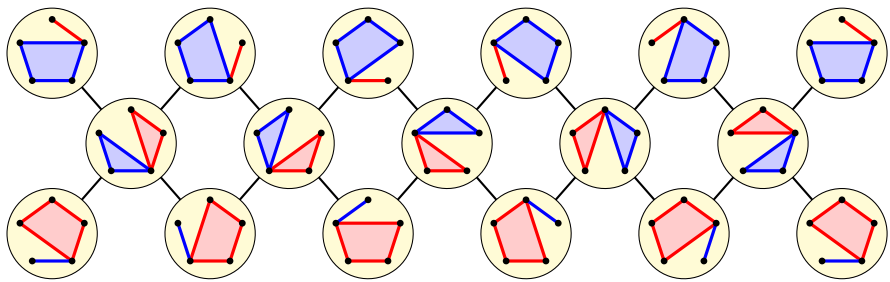
\begin{figure}
  \begin{tikzpicture}[scale=.3]
    \foreach \i in {0,1,2,3,4} {
      \draw[thick] (\i*7cm,0)--(7cm+\i*7cm,8cm);
      \draw[thick] (\i*7cm,8cm)--(7cm+\i*7cm,0);
    }
    \foreach \i in {0,1,2,3,4,5} {
      \begin{scope}[xshift=\i*7cm,rotate=\i*(-72)]
        \foreach \x in {1,2,...,10} {\coordinate (\x) at (306-72*\x:1.5cm);}
        \filldraw[YellowPoly] (0,0) circle (2cm);
        \draw[RedPoly] (2)--(3)--(4)--(5)--cycle;
        \draw[BluePoly] (1)--(5)--cycle;
        \foreach \x in {1,2,...,5} {\fill (\x) circle (1.5mm);}
      \end{scope}
    }
    \foreach \i in {0,1,2,3,4} {
      \begin{scope}[xshift=3.5cm+\i*7cm,yshift=4cm,rotate=\i*(-72)]
        \foreach \x in {1,2,...,10} {\coordinate (\x) at (306-72*\x:1.5cm);}
        \filldraw[YellowPoly] (0,0) circle (2cm);
        \draw[RedPoly] (3)--(4)--(5)--cycle;
        \draw[BluePoly] (1)--(2)--(5)--cycle;
        \foreach \x in {1,2,...,5} {\fill (\x) circle (1.5mm);}
      \end{scope}
    }
    \foreach \i in {0,1,2,3,4,5} {
      \begin{scope}[xshift=\i*7cm,yshift=8cm,rotate=\i*(-72)]
        \foreach \x in {1,2,...,10} {\coordinate (\x) at (306-72*\x:1.5cm);}
        \filldraw[YellowPoly] (0,0) circle (2cm);
        \draw[RedPoly] (3)--(4)--cycle;
        \draw[BluePoly] (1)--(2)--(4)--(5)--cycle;
        \foreach \x in {1,2,...,5} {\fill (\x) circle (1.5mm);}
      \end{scope}
    }
  \end{tikzpicture}
  \caption{The $15$ basic noncrossing hypertrees on $5$ vertices.
    Although displayed in a horizontal strip, the hypertrees on the
    left and the right are identical and should be identified to
    create an arrangement on a cylinder.  The lines indicate
    inclusions between their lower hyperedges or, equivalently,
    reverse inclusions between their upper
    hyperedges.\label{fig:two-hyperedges}}
\end{figure}

The poset of noncrossing hypertrees on a fixed number of vertices are
going to bijectively correspond to the poset of all dissections of a
polygon with twice as many vertices into even-sided subpolygons.  The
easy direction is from polygon dissections to noncrossing hypertrees.

\begin{rem}[From polygon dissections to noncrossing hypertrees]
 Given such a polygonal dissection of an even-sided polygon into
 even-sided subpolygons, one can create a noncrossing hypertree with
 half as many vertices as follows.  First assume that the vertices of
 the polygon have been alternately colored black and white and select
 the convex hull of the black vertices as the polygon for the
 noncrossing hypertree.  The convex hulls of the black vertices in
 each subpolygon are its hyperedges.  See Figure~\ref{fig:dissection}
 for an illustration.
\end{rem}

\begin{figure}
  \begin{tikzpicture}[scale=.7]
    \begin{scope}[xshift=-3cm]
      \foreach \x in {1,2,...,9} {\coordinate (\x) at (90-40*\x:1.5cm);}
      \filldraw[YellowPoly] (0,0) circle (2.5cm);
      \draw[BluePoly] (1)--(3)--(2)--cycle;
      \draw[BluePoly] (1)--(4)--(5)--cycle;
      \draw[BlueLine] (5)--(6);
      \draw[BluePoly] (5)--(7)--(8)--cycle;
      \draw[BlueLine] (8)--(9);
      \foreach \x in {1,2,...,9} {\fill (\x) circle (1mm);}
      \foreach \x in {1,2,...,9} {\draw (90-40*\x:2cm) node {\small
          \x};}
    \end{scope}
    \begin{scope}[xshift=3cm]
      \foreach \x in {1,2,...,9} {
        \coordinate (\x) at (90-40*\x:1.5cm);
        \coordinate (m\x) at (70-40*\x:1.5cm);
      }
      \filldraw[YellowPoly] (0,0) circle (2.5cm);
      \draw[very thick,color=red] (1)--(m1)--(2)--(m2)--(3)--(m3)--(4)--(m4)--(5)--(m5);
      \draw[very thick,color=red] (m5)--(6)--(m6)--(7)--(m7)--(8)--(m8)--(9)--(m9)--(1);
      \draw[very thick,color=red] (m6)--(5)--(m9)--(8) (m3)--(1);
      \foreach \x in {1,2,...,9} {
        \fill (\x) circle (1mm);
        \draw (90-40*\x:2cm) node {\small \x};
        \fill (m\x) circle (1mm);
        \fill[color=white] (m\x) circle (.7mm);
      }
    \end{scope}
  \end{tikzpicture}
  \caption{A noncrossing hypertree on $9$ vertices and the
    corresponding dissection of an $18$-gon into even-sided
    subpolygons.\label{fig:dissection}}
\end{figure}
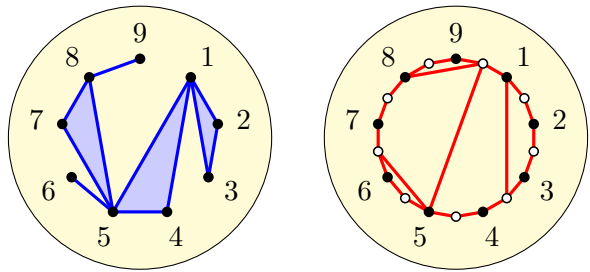

This construction is bijective because the process is reversible.

\begin{rem}[From noncrossing hypertrees to polygon dissections]
  Let $\tau$ be a noncrossing hypertree with hyperedge poset
  $\poset(\tau)$.  To create an even-sided polygon dissection of a
  polygon with twice as many vertices, first add a white dot in the
  middle of each boundary edge.  The diagonals of the dissection
  correspond to the covering relations of the hyperedge poset
  $\poset(\tau)$ as follows.  Let $e$ and $e'$ be hyperedges in $\tau$
  so that $e < e'$ is a covering relation in $\poset(\tau)$.  Using
  the Hasse diagram of $\poset(\tau)$ as a guide, systematically
  shrink all of the other covering relations, which amounts to
  iteratively merging various hyperedges.  By
  Proposition~\ref{prop:he-merge} the end result remains noncrossing
  and it now has only two hyperedges.  As remarked in
  Definition~\ref{def:basic}, this basic noncrossing hypertree is
  determined by the (black) vertex the two resulting hyperedges share
  and the (white vertex in the middle of) the boundary edge that is
  missing.  More concretely, for each vertex $v$ and for each pair of
  adjacent hyperedges that contain $v$ there is a unique midpoint of a
  boundary edge that is visible from $v$ looking between these two
  hyperedges.  The diagonal connecting this black vertex $v$ to this
  visible white vertex in the middle of the boundary edge is added.
  See Figure~\ref{fig:dissection}.  Because the endpoints have
  different colors the subpolygons are even-sided and these diagonals
  must be weakly noncrossing because when diagonals cross their black
  endpoints are in a different connected component from the black
  vertices between their white endpoints, contradicting the connected
  nature of the hypertree.
\end{rem}

A polygon dissection is determined by the set of pairwise weakly
noncrossing diagonals and these sets form a poset using inclusion.
The constructions described above establish the following.

\begin{thm}[Noncrossing hypertrees and polygon dissections]\label{thm:ncht-poly}
  There is a natural order-reversing bijection between the poset of
  noncrossing hypertrees on a fixed number of vertices and the poset
  of dissections of an even-sided polygon with twice as many
  vertices into even-sided subpolygons through the addition of
  pairwise noncrossing diagonals.
\end{thm}

Under this bijection the process of merging two adjacent hyperedges as
described in Proposition~\ref{prop:he-merge} corresponds to removing a
diagonal from the dissection.  Since any subset of the diagonals can
be removed, the upper intervals in the noncrossing hypertree poset are
Boolean lattices and the dual of the noncrossing hypertree poset is
the face poset of a simplicial complex.

\begin{defn}[Noncrossing hypertree complex]\label{def:ncht-complex}
  The dual of the noncrossing hypertree poset $\nchtrees_{n+1}$ is the
  face lattice of an $n-2$ dimensional simplicial complex
  $\complex(\nchtrees_{n+1})$ that I call the \emph{noncrossing
    hypertree complex}.  It has a vertex for each of the $n^2-1$ basic
  noncrossing hypertrees, or equivalently, a vertex for every diagonal
  of an alternately colored $(2n+2)$-sided polygon connecting a black
  vertex and a white vertex.  Two basic noncrossing hypertrees are
  \emph{compatible} and connected by an edge if and only if the
  corresponding diagonals are weakly noncrossing.  More generally
  there is a simplex, naturally labeled by a noncrossing hypertree,
  for every collection of pairwise compatible basic noncrossing
  hypertrees.  The maximal simplices correspond to noncrossing trees.
  A noncrossing tree on $n+1$ vertices has $n$ edges, so its hyperedge
  poset has $n$ vertices and $n-1$ covering relations and the maximal
  dimensional simplices, or \emph{tree chambers} have $n-1$ vertices
  and dimension $n-2$.  Table~\ref{tbl:ncht} shows the $f$-vector and
  reduced euler characteristic of the noncrossing hypertree complex
  for small values of $n$.  Recall that the number $f_i$ counts the
  number of $i$-dimensional simplices (with the dimension of the empty
  set being $-1$).  For the noncrossing hypertree complex this number
  is equal to the number of noncrossing hypertrees with $i+2$
  hyperedges.  Also note that the reduced euler characteristic of this
  complex is a signed Catalan number.  For further information see the
  sequence \texttt{oeis:A102537} in the Online Encyclopedia of Integer
  Sequences.
\end{defn}

\begin{table}
  \begin{center}
    $\begin{array}{|c|rrrrrrrr|r|}
      \hline
      n & f_{-1} & f_0 & f_1 & f_2 & f_3 & f_4 & f_5 & f_6 & \reuler\\
      \hline \hline
      2 & 1 & 3 &&&&&&  & 2 \\ 
      3 & 1 & 8 & 12 &&&&&  & -5\\
      4 & 1 & 15 & 55 & 55 &&&& & 14\\
      5 & 1 & 24 & 156 & 364 & 273 &&& & -42\\
      6 & 1 & 35 & 350 & 1400 & 2380 & 1428 && & 132\\
      7 & 1 & 48 & 680 & 4080 & 11628 & 15504 & 7752 && -429\\
      8 & 1 & 63 & 1197 & 9975 & 41895 & 92169 & 100947 & 43263 & 1430\\
      \hline
      \end{array}$
  \end{center}
\caption{The $f$-vector and reduced euler characteristic of
  $\complex(\nchtrees_{n+1})$ for small values of
  $n$.\label{tbl:ncht}}
\end{table}

Theorem~\ref{thm:ncht-poly} also means that the noncrossing hypertree
complex can be identified as a generalized cluster complex.

\begin{rem}[Generalized cluster complexes]\label{rem:gen-cluster}
  Generalized cluster complexes are a family of simplicial complexes
  with one complex defined for every finite Coxeter group and for
  every integer $m\geq 1$.  They were introduced for types $A$ and $B$
  by Tzanaki \cite{Tz06} and for all types by Fomin and Reading
  \cite{FoRe05}.  Since their introduction in 2005 they have been
  studied extensively \cite{AthTz06, Kr06, Tz08, AthTz08, FiKaTz13}.
  The specific case of even-sided dissections of an even-sided polygon
  corresponds to the generalized cluster complex of type $A$ with
  $m=2$.
\end{rem}

Although generalized cluster complexes have already been investigated,
different aspects of the noncrossing hypertree complex are visible
when its simplices are labeled by noncrossing hypertrees. In
particular the homeomorphism established in \S\ref{sec:top-geo}
between the noncrossing hypertree complex and the noncrossing
partition link is new and its proof relies heavily on the structure of
noncrossing hypertrees.

\section{Boolean Lattices, Cubes and Spheres\label{sec:bldg-ortho}}

This section reviews how to construct a simplicial sphere from a
Boolean lattice using the geometric shapes called orthoschemes.  The
material is elementary and well-known, but it is included to establish
notation for later constructions.  For a more detailed treatment of
orthoschemes and links see \cite{BrMc10}.

\begin{defn}[Boolean lattices]\label{def:boolean} 
  Let $S$ be a finite set of size $n$.  Its subsets under inclusion
  form the \emph{Boolean lattice of rank $n$} denoted $\bool(S)$ or
  $\bool_n$ when only the size of $S$ is relevant.  It contains, of
  course, many subposets that are isomorphic to smaller Boolean
  lattices and those with the same bounding elements are of particular
  interest.  For every partition of the set $S$ into $k$ blocks, there
  is a Boolean sublattice of type $\bool_k$ inside $\bool_n$ whose
  elements are unions of blocks of the partition.  I call these
  \emph{special Boolean sublattices}.
\end{defn}

\begin{defn}[Chains and weak linear orderings]\label{def:chains}
  If $S' \subset S''$ are two elements in $\bool(S)$, then the
  \emph{label} of this pair is the set of new elements, those in $S''$
  but not in $S'$.  For every chain $\emptyset=S_0 \subset S_1 \subset
  \cdots \subset S_k = S$ in $\bool(S)$ between the bounding elements,
  the labels on the adjacent elements in the chain form a partition of
  $S$. In fact, the collection of all such chains is in bijection with
  the \emph{ordered partitions of $S$}, by which I mean a partition of
  $S$ together with a linear ordering of its blocks.  An alternative
  way to encode an ordered partition is as a function from $S$ to any
  linearly ordered set such as the reals.  The blocks of the partition
  are determined by the elements with the same image and the ordering
  of the blocks by the linear ordering of these images.  Such a
  function is called a \emph{weak linear ordering of the set $S$} and
  two such functions are considered to be the same weak linear
  ordering when they determine the same partition and the same
  ordering of its blocks.  A (strict) \emph{linear ordering of the set
    $S$} is a weak linear ordering where the function is injective, or
  equivalently the partition is trivial. Thus chains in $\bool(S)$
  between the bounding elements bijectively correspond to weak linear
  orderings of $S$ and the maximal chains in $\bool(S)$ correspond to
  the strict linear orderings.
\end{defn}

The geometry of a Boolean lattice is essentially that of a cube.

\begin{defn}[Cubes and orthoschemes]\label{def:cubes}
  Let $S$ be a set with $n$ elements and identify the elements of $S$
  with the standard orthonormal basis of $\R^n$.  The elements of
  $\bool(S)$ can then be identified with the $2^n$ corners of an
  $n$-dimensional cube in $\R^n$ by sending each subset of $S$ to the
  sum of the corresponding basis vectors.  Thus the subset $S'$ is
  sent to the vector where every coordinate is $0$ or $1$ and the
  locations of the $1$'s indicating the elements in the subset. One of
  these points is less than another in the ordering if and only if the
  vector from the first to the second has nonnegative coordinates.
  For each maximal chain, the convex hulls of the corresponding $n+1$
  points in $\R^n$ with its standard euclidean geometry is a shape
  called an \emph{orthoscheme}, or more precisely a \emph{standard
    $n$-orthoscheme} in the language of
  \cite{BrMc10}. Figure~\ref{fig:orthoscheme} shows a $3$-dimensional
  orthoscheme.  Since there are $n!$ ways to linearly order the
  elements of $S$, the $n$-dimensional cube is subdivided into $n!$
  standard $n$-orthoschemes.
\end{defn}

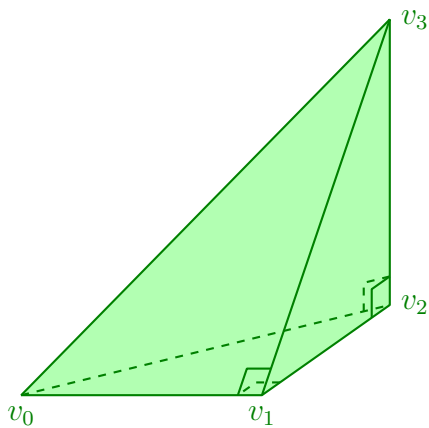
\begin{figure}
  \begin{tikzpicture}[scale=1,thick]
    \def\ux{.32cm} \def\vx{.17cm} \def\vy{.12cm} \def\wy{.38cm}
    \def\sca{10}
    \def\scb{1.414}
    \begin{scope}[GreenPoly]
      \filldraw (0,0) coordinate[label=below:$v_0$](t0) --
      ++(\sca*\ux,0) coordinate[label=below:$v_1$](t1) -- 
      ++(\sca*\vx,\sca*\vy) coordinate[label=right:$v_2$](t2) --
      ++(0,\sca*\wy) coordinate[label=right:$v_3$](t3) -- cycle;
      \draw[dashed] (0,0)--(t2);
      \draw (t1)--(t3);
      \drawAngle{(t2)}{(-1*\scb*\vx,-1*\scb*\vy)}{(0,\wy)};
      \drawAngleD{(t2)}{(-.707*\vx-.707*\ux,-.707*\vy)}{(0,\wy)};
      \drawAngle{(t1)}{(.707*\vx,.707*\vy+.707*\wy)}{(-\ux,0)};
      \drawAngleD{(t1)}{(\scb*\vx,\scb*\vy)}{(-\ux,0)};
    \end{scope}
  \end{tikzpicture}
  \caption{A $3$-dimensional orthoscheme.\label{fig:orthoscheme}}
\end{figure}

The process that converts a Boolean lattice into a metric simplicial
subdivided cube can be extended to other types of posets.

\begin{defn}[Orthoscheme complex]\label{def:ortho-complex}
  The \emph{order complex} of a poset is the simplicial complex formed
  by its chains with subchains as subsimplices: the elements, viewed
  as $1$-element chains are the vertices, the $2$-element chains are
  the edges, etc.  The \emph{orthoscheme complex} of a bounded graded
  poset is a metric version of its order complex.  The simplex
  corresponding to each maximal chain is given the same metric as on
  the maximal chains in the geometric version of the Boolean lattice,
  namely, that of an orthoscheme.  Concretely, the length of the edge
  connecting two vertices is the square root of the difference in the
  levels of the corresponding elements of the poset and, more
  generally, the metric on a simplex labeled by a chain is that of the
  unique euclidean simplex with these specified edge lengths.
\end{defn}

The orthoscheme complex of a Boolean lattice is, of course, the
simplicially subdivided cube described in Definition~\ref{def:cubes}.
When considering the orthoscheme complex of a poset with one or more
bounding elements, the resulting space is always contractible.  Thus,
it is usually more interesting topologically to consider the link of
the vertex or edge spanned by the bounding element(s) of the poset.

\begin{defn}[Link of a poset]\label{def:link}
  The \emph{link of a face of a euclidean polytope} is the spherical
  polytope formed by the set of unit vectors based at a point in the
  interior of the face that (1) are perpendicular to the affine span
  of the face and such that (2) the point plus some small positive
  scalar multiple of the unit vector remains in the polytope.  For
  example, the link of an edge in a euclidean tetrahedron is a
  spherical arc whose length is equal to the dihedral angle (in
  radians) along this edge.  The spherical polytope that results is
  independent of the point chosen in the interior of the face.  When
  the euclidean polytope is a euclidean simplex, the links are
  spherical simplices.  The \emph{link of a simplex in a piecewise
    euclidean simplicial complex} is the piecewise spherical
  simplicial complex formed by gluing together the spherical simplices
  that are the link of this simplex in each of the euclidean simplices
  that contain it as a face in the obvious fashion.  See \cite{BrMc10}
  or \cite{BrHa99} for more precise definitions.  In the concrete case
  where $P$ is a bounded graded poset, the edge connecting its two
  bounding elements in its orthoscheme complex is the unique longest
  edge and in \cite{BrMc10} we called it the \emph{long diagonal}.
  The \emph{link of the poset $P$} is the piecewise spherical
  simplicial complex that is the link of this long diagonal edge in
  its orthoscheme complex.
\end{defn}

From this construction, the following is immediate.

\begin{prop}[Links and subposets]\label{prop:link-subposet}
  If $P$ is a bounded graded poset and $Q$ is an induced subposet with
  the same bounding elements and the same grading, then $\link(Q)$
  is a simplicial subcomplex of $\link(P)$.  
\end{prop}

The link of a Boolean lattice is a simplicial sphere also known as a
Coxeter complex of type~$A$.

\begin{defn}[Coxeter complex of type $A$]\label{def:coxeter-complex}
  From the geometry of its orthoscheme complex it is clear that
  $\link(\bool_n)$ is a metric simplicial sphere $\sph^{n-2}$ with
  $n!$ top-dimensional simplices called \emph{chambers}.  That these
  simplices are all isometric can be seen from the action of the
  symmetric group on $\R^n$ by permuting coordinates which acts
  transitively on the standard $n$-orthoschemes in the cube and
  transitively on the chambers in $\sph^{n-2}$.  Because the symmetric
  group $\sym_n$ is the Coxeter group of type $A_{n-1}$, the sphere
  $\link(\bool_n)$ is called a \emph{Coxeter complex of type
    $A_{n-1}$} and the common shape of its chambers is a spherical
  simplex also known as the \emph{Coxeter shape of type~$A_{n-1}$}.
  Also, note that the simplices in this sphere correspond to chains in
  $\bool_n$ between its bounding elements and thus to weak linear
  orderings of the set $S$ used to construct $\bool_n$.  In fact, the
  coordinates of a point can be reinterpreted as a function from $S$
  to the reals and thus as a weak linear ordering of the set $S$.  The
  interiors of the simplices correspond to the sets of points in the
  sphere whose coordinate functions determine the same weak linear
  ordering of $S$.
\end{defn}

These Coxeter complexes have many simplicial subspheres.

\begin{prop}[Subspheres and special sublattices]\label{rem:subspheres}
  Let $S$ be a set with $n$ elements.  For every partition $\sigma$ of
  $S$ with $k$ blocks, there is a corresponding simplicial subcomplex
  of $\link(\bool_n)$ isometric to $\sph^{k-2}$.  It has the
  simplicial structure of a Coxeter complex of type $A_{k-1}$ but the
  metric might be slightly different.  Moreover, the top-dimensional
  simplices in this subsphere are precisely those correspoonding to
  the $k!$ orderings of the blocks of $\sigma$.
\end{prop}

\begin{proof}
  The partition $\sigma$ determines a special Boolean sublattice of
  type $\bool_k$ and the corresponding vertices of the $n$-cube are
  precisely those living in the subspace defined by setting the
  coordinates $x_i$ and $x_j$ equal if and only if $i$ and $j$ belong
  to the same block $\sigma$.  The intersection of this
  $k$-dimensional subspace with the Coxeter complex determines the
  $\sph^{k-2}$ sphere.  As a simplicial complex it is the link of
  $\bool_k$ but the geometry of the subcomplex is slightly distorted
  when compared to the standard Coxeter complex of this type because
  the embedding into $\R^n$ depends on the number of elements in each
  block.  The final assertion follows from the description of
  simplices in terms of coordinates given in
  Definition~\ref{def:coxeter-complex}.
\end{proof}

The simplicial subspheres of $\link(\bool_n)$ determined by a
partition in this way are called \emph{special subspheres}.  The
following result is nearly immediate from the construction.

\begin{cor}[Special subspheres]
  Every simplex in $\link(\bool_n)$ belongs to a unique special
  subsphere of the same dimension.
\end{cor}

\begin{proof}
  Let $S$ be the set of size $n$ used to construct $\bool_n$.  A
  simplex of dimension $k-2$ corresponds to an ordered partition
  $\sigma$ of $S$ with $k$ blocks and this partition then determines a
  special subsphere which includes the simplex.  Uniqueness follows
  from the explicit description of the top-dimensional simplices in
  the special subsphere.
\end{proof}

The final remarks in this section concern the various hemispheres in
the Coxeter complex and the ways in which they intersect.

\begin{defn}[Roots and hemispheres]\label{def:roots-hemis}
  Let $S$ be a set of size $n$ and let $\R^n$ have a standard
  orthonormal basis $\{\epsilon_i\}$ indexed by the elements $i$ in
  $S$.  The vector $v_{ij} := $ $\epsilon_i - \epsilon_j$ is called a
  \emph{root} and the set of all roots is the \emph{root system for
    $\sym_n$}.  The span of the root system is the subspace $\R^{n-1}$
  containing the Coxeter complex $\sph^{n-2}$ on which $\sym_n$ acts.
  For each choice of $i$ and $j$ in $S$ the inequality $x_i \geq x_j$
  defines a closed \emph{half-space} in $\R^n$ that becomes a closed
  \emph{hemisphere} $H_{ij}$ in $\sph^{n-2}$ characterized as the set
  of unit vectors in $\sph^{n-2}$ that form a nonobtuse angle with the
  root $v_{ij}$.  In this context, the root $v_{ij}$ is the
  \emph{pole} of $H_{ij}$.  The hemisphere $H_{ij}$ is also the union
  of the chambers that correspond to linear orderings of $S$ where $i$
  occurs before $j$.
\end{defn}

\begin{prop}[Intersecting hemispheres]\label{prop:intersections}
  Let $T$ be a directed tree with vertex set $S$ of size $n$ and
  directed edges.  If for each directed edge from $i$ to $j$ in $T$,
  $H_{ij}$ is the correspondng hemisphere in $\link(\bool(S))$ with
  pole $v_{ij}$, then the intersection of these hemispheres is a union
  of chambers that form a single spherical simplex.
\end{prop}

\begin{proof}
  Because $T$ has no cycles, the set of poles $v_{ij}$ is linearly
  independent and because $T$ is connected, it spans the subspace
  $\R^{n-1}$ containing the sphere $\sph^{n-2}$.  In particular these
  roots form a basis of this root subspace.  From this it quickly
  follows that the intersection of the closed hemispheres is a
  spherical simplex.
\end{proof}

\section{The Noncrossing Partition Link\label{sec:ncp-link}}

This section establishes basic facts about the piecewise spherical
simplicial complex that I call the noncrossing partition link with an
emphasis on its connections to spherical buildings, noncrossing
hypertrees and noncrossing hyperforests.

\begin{defn}[Noncrossing partition link]
  The \emph{noncrossing partition link} $\link(\ncparts_{n+1})$ is the
  link of the long diagonal edge in the orthoscheme complex of the
  noncrossing partition lattice $\ncparts_{n+1}$
  (Definition~\ref{def:ortho-complex}).  As a simplicial complex it is
  the order complex of the noncrossing partition lattice with both
  bounding elements removed.  The top-dimensional simplices in the
  noncrossing partition link are called \emph{partition chambers}.
\end{defn}

The noncrossing partition link can also be viewed as a subcomplex of
the spherical building associated to a linear subspace poset.  A
spherical building is a very special type of highly symmetric
simplicial complex with piecewise spherical metrics on its simplices.
Even though the only spherical buildings used in this article are
those of type $A$ (whose structure can be directly described without
mentioning the general theory), a few general remarks are in order.

\begin{defn}[Spherical buildings]
  A \emph{spherical building}, roughly speaking, is a metric
  simplicial complex that can viewed as a union of unit spheres of the
  same dimension where the simplicial structure on each sphere is that
  of a Coxeter complex for a fixed finite Coxeter group $W$.  These
  top-dimensional spheres are called \emph{apartments} and the
  top-dimensions simplices are called \emph{chambers}.  In a building,
  any two chambers belong to a common apartment and the isometry group
  of the building acts transitively on its chambers and transitively
  on its apartments.  The \emph{type of the building} is the type of
  the finite Coxeter group whose Coxeter complex is the model for the
  simplicial structure on each apartment.
\end{defn}

Since the link of a Boolean poset is a Coxeter complex of type $A$,
spherical buildings of type $A$ can be constructed from posets with
lots of maximal Boolean subposets.

\begin{defn}[Maximal Boolean subposets]\label{def:max-boolean}
  Let $P$ be a bounded graded poset.  A \emph{maximal Boolean subposet
    of $P$} is a subposet $Q$ isomorphic to a Boolean lattice
  $\bool_n$ with the same bounding elements and the same grading as
  $P$.  By Proposition~\ref{prop:link-subposet} each maximal Boolean
  subposet corresponds to a sphere in $\link(P)$.  This also means
  that any special Boolean sublattice inside a maximal Boolean
  subposet inside $P$ gives rise to a smaller dimensional sphere
  inside this sphere in the link of $P$.
\end{defn}

Spherical buildings of type $A$ are built out of linear subspaces.

\begin{defn}[Linear subspaces]\label{def:linear-subspaces}
  Let $V$ be a finite dimensional vector space over a field $\F$.  The
  \emph{linear subspace poset} $\linear(V)$ is the set of all linear
  subspaces of $V$ under inclusion.  It is a bounded graded self-dual
  poset and it has many maximal Boolean subposets.  For example, a
  maximal Boolean subposet of $\linear(V)$ can be constructed from any
  basis $\mathcal{B}$ by taking the span of every subset of
  $\mathcal{B}$.  More precisely, this construction depends less on
  the basis itself than on the lines the basis vectors determine.  The
  link of the linear subspace poset $\linear(V)$ is an example of a
  spherical building of type $A$.  Its apartments are precisely those
  spheres derived from maximal Boolean subposets of $\linear(V)$
  constructed from linearly independent spanning sets of lines.  As an
  example, consider the $3$-dimensional vector space $V$ over the
  field $\F_2$ with only two elements.  The corresponding spherical
  building is the Heawood graph shown in Figure~\ref{fig:bldg} where
  every edge is a metric spherical arc of length $\frac{\pi}{3}$.  The
  various combinatorial hexagons in the graph are thus unit circles
  and these are its apartments.
\end{defn}

\begin{figure}
  \begin{tikzpicture}
    \draw[thick] \poly{14}{2cm};
    \foreach \x in {1,2,...,14} {\coordinate (\x) at (360/14*\x:2cm);}
    \draw[thick] (2)--(7) (4)--(9) (6)--(11) (8)--(13);
    \draw[thick] (10)--(1) (12)--(3) (14)--(5);
    \foreach \x in {1,2,...,14} {\filldraw[black] (\x) circle (2pt);}
    \foreach \x in {2,4,...,14} {\filldraw[white] (\x) circle (1.4pt);}
   \end{tikzpicture}
  \caption{A spherical building of type $A_2$.  The seven black dots
    correspond to the $1$-dimensional subspaces of the $3$-dimensional
    vector space over $\F_2$ and the seven white dots to the
    $2$-dimensional subspaces.  The edges are meant to represent
    spherical arcs of length~$\frac{\pi}{3}$.\label{fig:bldg}}
\end{figure}
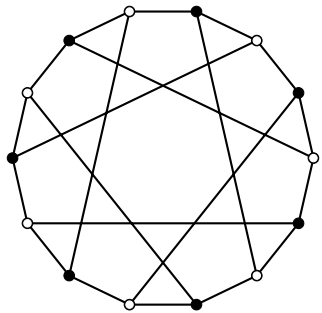

The noncrossing partition lattice embeds into a linear subspace poset.

\begin{rem}[Noncrossing partitions and linear subspaces]
  If the vertices of a convex polygon are identified with the
  coordinates of a vector space, then every noncrossing partition (and
  in fact every partition) naturally describes a linear subspace
  defined by equating the coordinates belonging to each nontrivial
  block.  For example, the partition $\sigma =
  \{\{1,3,4\},\{2\},\{5,6,7,8,9\}\}$ is sent to the $3$-dimensional
  subspace of tuples $(x_1,x_2,\ldots,x_9)$ such that $x_1 = x_3 =
  x_4$ and $x_5 = x_6 = x_7 = x_8 = x_9$.  This map embeds the
  noncrossing partition lattice $\ncparts_{n+1}$ into the \emph{dual}
  of the linear subspace poset $\linear(\F^{n+1})$ because $\sigma <
  \tau$ implies that the subspace associated to $\sigma$ contains the
  subspace associated to $\tau$.  The order reversal is of little
  consequence since both $\ncparts_{n+1}$ and $\linear(\F^{n+1})$ are
  self-dual posets.  Of more concern is the fact that the top element,
  the noncrossing partition that I call the Coxeter hypertree, is sent
  to a $1$-dimensional subspace rather than the trivial subspace.  To
  remedy this intersect each subspace with the codimension~$1$
  subspace where vectors have coordinate sum~$0$.  The new map embeds
  $\ncparts_{n+1}$ into the dual of $\linear(\F^n)$.
\end{rem}

This embedding means that the noncrossing partition link can be viewed
as a simplicial subcomplex of the corresponding spherical building
(Proposition~\ref{prop:link-subposet}).  In fact, Tom Brady and I
noticed around 2003 that the noncrossing partition link can be viewed
as a union of apartments in this spherical builiding.  The first
published proof of this result is in \cite[Proposition~3.25]{HKS}.

\begin{prop}[Noncrossing apartments]\label{prop:nc-apartments}
  The maximal Boolean subposets of the noncrossing partition lattice,
  and thus the only apartments contained in the noncrossing partition
  link, are in natural bijection with the set of noncrossing trees.
\end{prop}

\begin{proof}
  Since the apartments of the spherical building inside the
  noncrossing partition lattice correspond, by definition, to its
  maximal Boolean subposets, it is sufficient to focus on these.  Let
  $\tau$ be a noncrossing tree.  For every subset of edges of $\tau$,
  the resulting noncrossing forest determines a noncrossing partition
  (Proposition~\ref{prop:part-nchf}) and these elements form a maximal
  Boolean subposet.  Conversely, in any maximal Boolean subposet of
  $\ncparts_{n+1}$ there are exactly $n$ elements that cover the
  bottom element, each of which is represented by a single edge
  connecting two vertices of the convex polygon.  These edges are
  necessarily weakly noncrossing because crossing edges have a join
  that is higher than it should be in a maximal Boolean subposet.  And
  since the join of all $n$ elements is the top element of
  $\ncparts_{n+1}$, these edges must connect all $n+1$ vertices.
  Since $n$ is the minimum number of edges needed to do so, there are
  none left over to form a cycle and their union is a tree.  Finally,
  these two procedures are inverses of each other, thus establishing
  the bijection.
\end{proof}

\begin{defn}[Apartments and spheres]\label{def:apartments-spheres}
  For each noncrossing tree $\tau$ the maximal Boolean subposet
  constructed in Proposition~\ref{prop:nc-apartments} corresponds to a
  top-dimensional sphere in the noncrossing partition link called a
  \emph{noncrossing apartment} and denoted $\apartment(\tau)$.  A
  similar construction can be used with noncrossing hypertrees.  For
  each noncrossing hypertree $\tau$, there is a special Boolean
  sublattice constructed from the noncrossing partitions that are the
  convex hulls of the connected components of the noncrossing
  hyperforests formed by subsets of the hyperedges of $\tau$.  And
  this special Boolean subposet corresponds to a subsphere in the link
  called the \emph{sphere of $\tau$} and denoted $\sphere(\tau)$.
\end{defn}

The fact that the noncrossing partition link is a union of these
noncrossing apartments will be an immediate consequence of
Lemma~\ref{lem:po-spheres}.  The next remark relates the construction
of an apartment from a noncrossing tree to the construction of an
apartment in the spherical building from a basis of the vector space.
It is not needed later, but it is included to clarify how the
constructions correspond.

\begin{rem}[Noncrossing apartments and bases]
  Let $\tau$ be a noncrossing tree on $n+1$ vertices.  Subsets of the
  edges of $\tau$ determine a maximal Boolean subposet of
  $\ncparts_{n+1}$ which is sent to a maximal Boolean subposet of
  $\linear(\F^n)$.  Because of the order reversal of the embedding,
  the lines in the vector space $V = \F^n$ that determine the
  corresponding apartment come from the elements of the Boolean
  subposet in $\ncparts_{n+1}$ covered by the top element, the
  noncrossing partitions with exactly $2$ blocks determined by
  removing a single edge from $\tau$.  Let $\sigma_e$ be the
  noncrossing partition formed when the edge $e$ is removed and let
  $A$ and $B$ be the blocks of sizes $k$ and $\ell$ it contains.  From
  $\sigma_e$ we construct a vector $v_e$ with a value of $\ell$ in
  each of the coordinates associated with the $k$ elements in block
  $A$ and a value of $-k$ in each of the coordinates associated with
  the $\ell$ elements in block $B$.  This vector $v_e$ in $\F^{n+1}$
  has coordinate sum $0$ and the line it spans is the image of
  $\sigma_e$ under the embedding.  Note that if the roles of blocks
  $A$ and $B$ are reversed, the vector produced is $-v_e$ and the line
  remains the same.  The lines associated to these $2$ block
  partitions are linearly independent and span the subspace of vectors
  with coordinate sum~$0$.  In fact, since $v_e$ is the only vector
  that assigns different values to the coordinates associated with its
  endpoints, its contribution to a vector with coordinate sum~$0$, is
  determined by the difference between these coordinate values.
  Subtracting off these necessary contributions leaves a vector where
  all coordinates are equal and the coordinate sum is unchanged.
\end{rem}

\section{Permutations and Reduced Products\label{sec:perm-product}}

Ever since the work of Brady and Watt \cite{BradyWatt02} and Bessis
\cite{Be03} more than a decade ago, the modern way to view the
noncrossing partition lattice is as a portion of the symmetric group
between the identity and a Coxeter element with respect to reflection
length, and the proof of the homeomorphism between the noncrossing
partition link and the noncrossing hypertree complex uses these ideas.
The first step is to turn hyperedges into irreducible permutations and
chains in the noncrossing partition lattice into reduced products.

\begin{defn}[Hyperedges and irreducible permutations]\label{def:he-ip}
  An \emph{irreducible permutation} is one represented by a single
  cycle of length $2$ or more.  The name refers to the fact that these
  permutations are Coxeter elements for irreducible subroot systems of
  the root system of the symmetric group.  There is a map from
  irreducible permutations to hyperedges that sends the permutation to
  the set of elements moved by its single cycle, but it is obviously
  not one-to-one since $(k-1)!$ irreducible permutations are sent to
  each $k$-element hyperedge.  In the other direction, there is an
  injective map that sends each hyperedge to the unique irreducible
  permutation where the elements in the single cycle are listed in the
  order they appear in the boundary cycle of their convex hull when
  traversed clockwise.  When the standard vertex labeling is used this
  corresponds to their natural linear order as integers.  Under these
  maps the irreducible permutation $(1,3,5,2)$ is sent to the
  hyperedge $\{1,2,3,5\}$ and this hyperedge is sent to the
  permutaiton $(1,2,3,5)$ (when the standard vertex labeling is used).
  The image of the unique hyperedge in the Coxeter hypertree on $n+1$
  vertices with the standard vertex labeling is the permutation
  $c=(1,2,\ldots,n+1)$, a notation that reflects the role of $c$ as a
  \emph{Coxeter element} in symmetric group $\sym_{n+1}$.
\end{defn}

\begin{defn}[Noncrossing permutations]
  If $\sigma$ is a noncrossing partition, the irreducible permutations
  of its hyperedges pairwise commute and their well-defined product is
  called a \emph{noncrossing permutation}. Each noncrossing partition
  is identified with its noncrossing permutaiton. The partition
  $\sigma = \{\{1,3,4\}, \{2\}, \{5,6,7,8,9\}\}$, for example, becomes
  the permutation $(1,3,4)(5,6,7,8,9)$ when the standard vertex order
  is used.  With this identification, the noncrossing partition of a
  noncrossing hyperforest (Proposition~\ref{prop:part-nchf}) can now
  be viewed as the \emph{noncrossing permutation of a noncrossing
    hyperforest}.
\end{defn}

\begin{defn}[Reduced products]
  Let $W$ be a Coxeter group and let $R$ denote its set of
  reflections.  The \emph{reflection length} of a element $\sigma \in
  W$ is the length $\ell_R(\sigma)$ of the smallest list of
  reflections whose product is $\sigma$.  Some parts of the literature
  call this the \emph{absolute length} of $\sigma$.  In the symmetric
  group $\sym_{n+1}$, the Coxeter group of type $A_n$, the reflections
  are the transpositions, and in the spirit of
  Proposition~\ref{prop:degrees} if you expect every cycle to have
  length~$1$ then the reflection length of a permutation is the sum of
  its excess cycle lengths.  A product $\sigma_1 \sigma_2 \cdots
  \sigma_k = \sigma$ is called a \emph{reduced product} and a
  \emph{reduced factorization of $\sigma$} if the reflection length
  $\ell_R(\sigma)$ of the product is equal to the sum $\sum_{i=1}^k
  \ell_R(\sigma_i)$ of the reflection lengths of the factors.  This
  leads to a partial \emph{reflection order} on the elements of $W$:
  $\sigma' < \sigma$ if and only if $\sigma'$ is a factor in a reduced
  factorization of $\sigma$.
\end{defn}

Reduced products are the same as what Vivian Ripoll calls ``block
factorizations'' \cite{Ri11,Ri12}.  That term is not used here because
the word ``block'' might mistakenly lead one to assume that each
factor corresponds to a single block of a partition, but this is not
part of the definition and is only true when the factors are
irreducible permutations.  Reduced products have many easy-to-prove
properties.

\begin{prop}[Reduced products]\label{prop:rp}
  If $\sigma_1 \sigma_2 \cdots \sigma_k = \sigma$ is a reduced
  product, then so is the portion of the factorization $\sigma_i
  \sigma_{i+1} \cdots \sigma_j = \sigma'$ for all $1 \leq i \leq j
  \leq k$ and so is the factorization of $\sigma$ where the portion
  from $\sigma_i$ to $\sigma_j$ is replaced by its product $\sigma'$.
\end{prop}

In the definition of the partial order on elements of $W$, the element
$\sigma$ can be required to be the first or last factor in a reduced
factorization of $\sigma$ without changing the partial order because
of the following easy fact that follows from the closure of the set of
reflections under conjugation.

\begin{prop}[Rewriting reduced products]\label{prop:rewriting}
  Let $\sigma_1 \sigma_2 \cdots \sigma_k$ be a reduced
  factorization of $\sigma$ in a Coxeter group $W$. For any selection
  $1 \leq i_1 < i_2 < \cdots < i_j \leq k$ of positions there is a
  length~$k$ reduced factorization of $\sigma$ whose first $j$ factors
  are $\sigma_{i_1} \sigma_{i_2} \cdots \sigma_{i_j}$ and another
  length~$k$ reduced reduced factorization of $\sigma$ where these are
  the last $j$ factors in the factorization.
\end{prop}

The set of permutations below the Coxeter element $c$ in the
reflection order is exactly the set of noncrossing permutations and
their poset structure in this partial order is isomorphic to that of
the noncrossing partition lattice $\ncparts_{n+1}$
\cite{BradyWatt02,Be03}.  This immediately implies the following
result.

\begin{prop}[Factors in reduced products]\label{prop:reduced-factors}
  Let $\sigma_1 \sigma_2 \cdots \sigma_k$ be a reduced factorization
  of a permutation $\sigma$.  When $\sigma$ is a noncrossing
  permutation, every $\sigma_i$ is a noncrossing permutation and,
  equivalently, when one $\sigma_i$ is not noncrossing, $\sigma$ is
  not noncrossing.
\end{prop}

Another way to state this connection between partitions and
permutations is that the Hasse diagram of the noncrossing partition
lattice is a portion of the Cayley graph of the symmetric group.  More
precisely, the Hasse diagram of the poset $\ncparts_{n+1}$ is the same
as the union of the directed geodesic paths from the identity to the
Coxeter element $c = (1,2,\ldots,n+1)$ in the right Cayley graph of
the symmetric group $\sym_{n+1}$ generated by the set $R$ of all
reflections.  In particular, if $c = r_1 r_2 \cdots r_n$ is a reduced
factorization of the Coxeter element $c$ into reflections, then there
is a path of length $n$ in the Hasse diagram of $\ncparts_{n+1}$ from
the identity element at the bottom to the Coxeter element at the top
whose edges are labeled $r_1, r_2, \ldots, r_n$ in that order.  More
generally, for every pair of noncrossing permutations $\tau < \tau'$
there is a well-defined \emph{label} that is the unique noncrossing
permutation $\sigma$ solving the equation $\tau \sigma = \tau'$, and
the sequence of labels on a chain in the noncrossing partition lattice
describes a reduced factorization of the label assigned to its pair of
endpoints.

\begin{prop}[Reduced products and chains]\label{prop:rp-chains}
  For every pair of elements $\tau < \tau'$ in the noncrossing
  partition lattice $\ncparts_{n+1}$, the chains of length $k$ from
  $\tau$ to $\tau'$ are in one-to-one correspondence with the reduced
  factorizations of the unique element $\sigma$ such that $\tau\sigma
  = \tau'$ into $k$ factors.  Concretely the chain $\tau = \tau_0 <
  \tau_1 < \cdots < \tau_k = \tau'$ is paired with the reduced
  factorization $\sigma_1 \sigma_2 \cdots \sigma_k = \sigma$ if and
  only if the $\sigma_i$'s and $\tau_i$'s satisfy the equations:
  $\tau_0 \sigma_1 \cdots \sigma_i = \tau_i$ for $i = 1, \ldots, k$.
\end{prop}

\begin{proof}
  Given the endpoints and the $\sigma_i$'s one can solve for the
  $\tau_i$'s.  Conversely, given the $\tau_i$'s one can solve for the
  $\sigma_i$'s.
\end{proof}

\section{Reduced Products and Noncrossing Hyperforests\label{sec:rp-nchf}}

When the permutations of the hyperedges of a noncrossing hyperforest
$\tau$ are multiplied together, the result depends on the order of
multiplication unless, of course, $\tau$ is a noncrossing partition
whose hyperedge permutations pairwise commute.  This section explores
the impact that this order has on the result with a goal of
establishing a bijection between noncrossing hyperforests whose
hyperedges have been ``properly ordered'' in a sense made precise in
Definition~\ref{def:ordered} and reduced products of noncrossing
permutations into irreducible permutations (Theorem~\ref{thm:rp-hf}).
As a first step consider the product of two irreducible permutations
that are weakly disjoint but not disjoint.

\begin{lem}[Two irreducible permutations]\label{lem:two-ips}
  If $\sigma'$ and $\sigma''$ are irreducible permutations whose
  single cycles have exactly one element in common, then $\sigma'
  \sigma''$ is a reduced product and the product $\sigma$ is an
  irreducible permutation.  Moreover, the product $\sigma$ is a
  noncrossing permutation if and only if the factors $\sigma'$ and
  $\sigma''$ are noncrossing permutations that form a noncrossing
  hyperforest with $2$ hyperedges where $\sigma'$ is to the left of
  $\sigma''$ when viewed from the perspective of their common vertex.
\end{lem}

\begin{proof}
  Let $k+1$ and $\ell+1$ be the lengths of the nontrivial cycles of
  $\sigma'$ and $\sigma''$ respectively and let $a$ be the element
  they have in common.  If we write $(a,b_1,\ldots,b_k)$ for $\sigma'$
  and $(a,c_1,\ldots,c_\ell)$ for $\sigma''$, then
  $(a,b_1,\ldots,b_k)\cdot (a,c_1,\ldots,c_\ell) =
  (a,b_1,\ldots,b_k,c_1,\ldots, c_\ell) = \sigma$.  All of the
  assertions follow immediately from this explicit computation.
\end{proof}

When analyzing more complicated products it is convenient to attach a
linear ordering to the hyperforest under consideration.

\begin{defn}[Ordered hyperforests]\label{def:ordered}
  An \emph{ordered hyperforest} is a hyperforest $\tau$ together with
  a linear ordering of it hyperedges.  The product of the hyperedge
  permutations of $\tau$ in this prescribed order is called the
  \emph{permutation of $\tau$}.  This should not be confused with the
  noncrossing permutation of noncrossing hyperforest which is the
  permutation of the noncrossing partition formed from the blocks that
  are its connected components.  The relationship between these two
  permutations is explained in Lemma~\ref{lem:po-nc}.  When the
  hyperforest $\tau$ is a noncrossing hyperforest, it has a hyperedge
  poset (Definition~\ref{def:edge-posets}) and the ordering on $\tau$
  is \emph{proper} and $\tau$ is \emph{properly ordered} when the
  linear ordering of the hyperedges of $\tau$ is a linear extension of
  the partial ordering of its hyperedges recorded in its hyperedge
  poset.
\end{defn}

The next result extends Lemma~\ref{lem:two-ips} to all ordered
hyperforests.

\begin{lem}[Ordered hyperforest permutations]\label{lem:ohfp}
  The permutation of an ordered hypertree is an irreducible
  permutation of all of its vertices and the permutation of an ordered
  hyperforest has a cycle type determined by the sizes of its
  connected components.  In both cases these products are reduced
  products of the hyperedge permutations.
\end{lem}

\begin{proof}
  For hypertrees this follows from Lemma~\ref{lem:two-ips} and an easy
  induction.  For hyperforests it is enough to note that the hyperedge
  permutations in distinct connected components pairwise commute.
\end{proof}

The next two lemmas focus on turning a reduced product of irreducible
permutations into a properly ordered noncrossing hyperforest.

\begin{lem}[Reduced products and hyperforests]\label{lem:hyperforest}
  If $\sigma_1 \sigma_2 \cdots \sigma_k = \sigma$ is a reduced
  factorization of a permutation into irreducible permutations, the
  hyperedges of the factors form a hyperforest $\tau$.
\end{lem}

\begin{proof}
  Let $\tau$ be the hypergraph whose hyperedges correspond to the
  factors $\sigma_i$.  For each connected component of $\tau$, it is
  possible to use Proposition~\ref{prop:rewriting} and
  Proposition~\ref{prop:rp} to find a reduced product that only
  contains the factors $\sigma_i$ corresponding to the hyperedges in
  this component.  Let $\tau'$ be the connected hypergraph with a
  vertex set restricted to this component and with only these
  hyperedges.  The equation defining a reduced product implies that
  equality holds in the equation in Proposition~\ref{prop:edge-sizes}.
  Therefore $\tau'$ is a hypertree and $\tau$ is a hyperforest.
\end{proof}

\begin{lem}[Weakly noncrossing and properly ordered]\label{lem:wnc-and-po}
  Let $\tau$ be an ordered hyperforest whose vertices have been
  identified with the vertices of a convex polygon in the plane. If
  the permutation of $\tau$ is a noncrossing permutation then its
  hyperedges are pairwise weakly noncrossing and they are properly
  ordered.  In particular, if the permutation of $\tau$ is a
  noncrossing permutation then $\tau$ is a properly ordered
  noncrossing hyperforest.
\end{lem} 

\begin{proof}
  Let $\sigma = \sigma_1 \sigma_2 \cdots \sigma_k$ be the product of
  the hyperedge permutations of $\tau$ in the specified order.  By
  Lemma~\ref{lem:ohfp} this is a reduced factorization of $\sigma$.
  By Proposition~\ref{prop:rewriting} for any two hyperedge
  permutations $\sigma_i$ and $\sigma_j$ with $i<j$ there is another
  reduced factorization of $\sigma$ into $k$ permutations where
  $\sigma_i$ and $\sigma_j$ are its first two factors.  Next, in the
  altered reduced factorization of $\sigma$ replace $\sigma_i
  \sigma_j$ with their product $\sigma'$.  The result stays reduced
  (Proposition~\ref{prop:rp}) and by
  Proposition~\ref{prop:reduced-factors} $\sigma'$ must be a
  noncrossing permutation.  Because $\tau$ is a hyperforest, the
  hyperedges of $\sigma_i$ and $\sigma_j$ are either disjoint or have
  exactly one vertex in common.  If they are disjoint, $\sigma_i$ and
  $\sigma_j$ commute, their order is irrelevant and the fact that
  $\sigma'$ is a noncrossing permutation means that their convex hulls
  must be noncrossing in the strong sense of being completely
  disjoint.  If the hyperedges of $\sigma_i$ and $\sigma_j$ have a
  vertex in common, then Lemma~\ref{lem:two-ips} and the fact that
  $\sigma'$ is noncrossing imply that $\sigma_i$ and $\sigma_j$ are
  weakly noncrossing and properly ordered.
\end{proof}

The final two lemmas are used to turn a properly ordered
noncrossing hyperforest into a reduced product of irreducible
permutations.  

\begin{lem}[Proper orderings and commutations]\label{lem:po-com}
  Let $\tau$ is a properly ordered noncrossing hyperforest.  If there
  are disjoint hyperedges in $\tau$ that are consecutive in the
  ordering, then they are incomparable in the hyperedge poset of
  $\tau$ and the new ordered noncrossing hyperforest obtaining by
  switching the order on these two is also proper.
\end{lem}

\begin{proof}
  If $e$ and $e'$ are two hyperedges of $\tau$ that are disjoint but
  comparable in the ordering of the hyperedge poset, say with $e <
  e'$, it is because there is a sequence of covering relations $e =
  e_0 < e_1 < \cdots e_\ell = e'$ with $\ell>1$.  In particular, in
  the linear ordering of the hyperedges of $\tau$, $e$ and $e'$ are
  not consecutive because $e_1$ must occur between them.  Thus
  disjoint hyperedges that are consecutive in the ordering are
  incomparable in the ordering of the hyperedge poset and the rest is
  clear.
\end{proof}

The final lemma highlights the algebraic significance of being
properly ordered.

\begin{lem}[Proper orderings and noncrossing permutations]\label{lem:po-nc}
  The permutation of a properly ordered noncrossing hyperforest is a
  noncrossing permutation.
\end{lem}

\begin{proof}
  Let $\sigma_1 \sigma_2 \cdots \sigma_k =\sigma$ be the permutation
  of a properly ordered noncrossing hyperforest $\tau$ with $k$
  hyperedges.  When $\tau$ is a noncrossing partition there is nothing
  to prove, and note that this includes the case where $k=1$.  So
  suppose it is true for all properly ordered noncrossing hyperforests
  with fewer hyperedges and that $\tau$ is a not a noncrossing
  partition.  Let $\sigma_i$ and $\sigma_j$ be factors that do not
  commute because they contain a vertex in common and select $i < j$
  so that they are ``innermost'' in the sense that all other pairs of
  factors in the portion $\sigma_i \sigma_{i+1} \cdots \sigma_j$ are
  disjoint and commute.  Use the commutations to reorder the factors
  so that $\sigma_i$ and $\sigma_j$ are adjacent.  By
  Lemma~\ref{lem:two-ips} the product $\sigma_i \sigma_j$ is an
  irreducible permutation $\sigma'$ and by Lemma~\ref{lem:hyperforest}
  and Lemma~\ref{lem:wnc-and-po} the reordered factorization of
  $\sigma$ with $\sigma'$ in place of $\sigma_i \sigma_j$ has factors
  that correspond to the hyperedges of a properly ordered noncrossing
  hyperforest.  By Lemma~\ref{lem:two-ips}, factoring $\sigma'$ back
  into $\sigma_i$ and $\sigma_j$ splits the hyperedge of $\sigma'$
  into a pair of properly ordered hyperedges $\sigma_i$ and
  $\sigma_j$.  Finally, undoing the commutations only changes the
  ordering and not the underlying noncrossing hyperforest and by
  Lemma~\ref{lem:po-com} it remains properly ordered.
\end{proof}

\begin{thm}[Reduced products and noncrossing hyperforests]\label{thm:rp-hf}
  For each noncrossing permutation $\sigma$ there is a natural
  bijection between the set of reduced factorizations of $\sigma$ into
  irreducible permutations and the set of properly ordered noncrossing
  hyperforests $\tau$ with $\sigma$ as its noncrossing permutation.
\end{thm}

\begin{proof}
  Let $\sigma_1 \sigma_2 \cdots \sigma_k = \sigma$ be a reduced
  factorization of a noncrossing permutation $\sigma$ into irreducible
  permutations.  By Proposition~\ref{prop:reduced-factors} the factors
  $\sigma_i$ are noncrossing permutations.  In particular they are the
  permutations of the hyperedges to which they correspond.  By
  Lemma~\ref{lem:hyperforest} these hyperedges are form a
  hyperforest $\tau$ that by Lema~\ref{lem:wnc-and-po} is noncrossing
  and properly ordered.  In the other direction let $\tau$ be a
  properly ordered noncrossing hyperforest.  By definition, its
  permutation $\sigma$ is a product of irreducible permutations, by
  Lemma~\ref{lem:ohfp} this product is reduced, and by
  Lemma~\ref{lem:po-nc} the permutation $\sigma$ is a noncrossing
  permutation.
\end{proof}

The following corollary is an immediate consequence of
Theorem~\ref{thm:rp-hf} applied to the case where $\sigma$ is the
Coxeter element $c$ and the reduced factorizations considered have
maximal length.  These correspond to maximal chains in the noncrossing
partition lattice and to the top-dimensional simplices in the
noncrossing partition link that I call partition chambers.  Note that
the labels on a maximal chain are reflections which are automatically
irreducible.

\begin{cor}[Chambers and noncrossing trees]
  The partition chambers in the noncrossing partition link
  $\link(\ncparts_{n+1})$ are bijectively labelled by properly ordered
  noncrossing trees on $n+1$ vertices.
\end{cor}

Theorem~\ref{thm:rp-hf} can also be used to create labels on all the
simplices in the noncrossing partition link, but the labels on
arbitrary chains between the bounding elements of the noncrossing
partition lattice are merely noncrossing permutations and not
necessarily irreducible ones.  To accommodate this one can expand the
set of linear orderings.

\begin{defn}[Weak proper orderings]
  A weak linear ordering of a set $S$ was defined in
  Definition~\ref{def:chains} as a function $f$ from a set $S$ to some
  linearly ordered set.  When the set $S$ is replaced by a poset $P$
  there are more distinctions to be made depending on the extent to
  which the function respects the structure of the poset.  When the
  function $f$ is injective and $p < p'$ implies $f(p) < f(p')$, this
  is the usual notion of a \emph{linear extension of $P$}.  When the
  injectivity assumption is dropped but $p < p'$ still implies $f(p) <
  f(p')$ it is called a \emph{weak linear extension of $P$}.  Finally,
  when the only condition satisfied is that $p \leq p'$ implies $f(p)
  \leq f(p')$ it is called an \emph{extremely weak linear extension of
    $P$}.  A \emph{weakly ordered hyperforest} is a hyperforest with a
  weak linear ordering of its set of hyperedges. A \emph{weakly
    properly ordered noncrossing hyperforest} is a weakly ordered
  hyperforest where the weak ordering is a weak linear extension of
  its hyperedge poset.
\end{defn}

With these definitions in place, the following is immediate.

\begin{cor}[Simplices and noncrossing hypertrees]\label{cor:simplices}
  The simplices in the noncrossing partition link
  $\link(\ncparts_{n+1})$ are bijectively labelled by weakly properly
  ordered noncrossing hypertrees on $n+1$ vertices.
\end{cor}

The noncrossing hypertree with a weak proper ordering of its
hyperedges is called the \emph{standard name} of the corresponding
simplex in the noncrossing partition link.  As a final remark, I
should note that many of the results in this section are well-known
although they are usually stated in a very different language.  They
are closely related to the various results that descend from the early
work of Goulden and Jackson \cite{GoJa92}, from \cite{GoPe93} to the
recent article by Du and Liu \cite{DuLiu15}, that discuss permutation
factorizations using multi-noded trees, as well as the article by
Irving \cite{Ir09} that discusses minimal transitive factorizations.
Both of these concepts are closely related to hypertrees.  The closest
match is to the article \cite{BHN} which includes drawings of
noncrossing hypertrees.

\section{Theorem~\ref{main:geometry}: Topology and Geometry\label{sec:top-geo}}

This section completes the proof of Theorem~\ref{main:geometry} by
constructing a natural homeomorphism between the noncrossing hypertree
complex and the noncrossing partition link.  For the sake of
readability, the noncrossing hypertree complex
$\complex(\nchtrees_{n+1})$ is denoted $C$ and the noncrossing
partition link $\link(\ncparts_{n+1})$ is denoted $L$ throughout this
section.  At this point it should be quite clear that there is some
sort of a relationship between $C$ and $L$ since the top-dimensional
simplices in $C$ are labeled by noncrossing trees and the
top-dimensional simplices in $L$ are labeled by properly ordered
noncrossing trees and more general simplices are labeled by
noncrossing hypertrees and weakly properly ordered noncrossing
hypertrees.  The first step in the construction of a map from $C$ to
$L$ is to decide where to send the vertices of~$C$.

\begin{defn}[Vertex images]\label{def:vertices}
  Recall that the $(n+1)(n-1)$ basic noncrossing hypertrees on $n+1$
  vertices (Definition~\ref{def:basic}) are the vertices of the
  noncrossing hypertree complex $C$.  Let $\tau$ be one of these basic
  noncrossing hypertrees.  Its two hyperedges have a unique proper
  ordering and it thus corresponds to a unique vertex in the
  noncrossing partition link. The basic noncrossing hypertree $\tau$
  also corresponds to a chain in the noncrossing partition lattice
  with a unique element other than its endpoints at the bounding
  elements and this unique point is the noncrossing partition
  corresponding to the lower hyperedge.  In particular, the ordering
  of basic noncrossing hypertrees in an annular arrangement defined in
  Definition~\ref{def:basic} and shown in
  Figure~\ref{fig:two-hyperedges} is none other than the ordering of
  the corresponding elements inside the noncrossing partition lattice.
\end{defn}

The other simplices in $C$ are sent to unions of simplices in $L$, and
the rough idea is to send each simplex in $C$ labeled by a noncrossing
hypertree $\tau$ to the union of the simplices in $L$ labeled by all
proper orderings of this same noncrossing hypertree $\tau$.  What
needs to be shown is that these unions form simplicial shapes, that
they overlap in the required fashion, and that the resulting
continuous map is bijective.  As an introduction to the general proof,
consider the two simplicial complexes when there are $4$ vertices and
both complexes are graphs.

\begin{figure}
  \begin{tikzpicture}[scale=.5]
    \begin{scope}[xshift=-4cm]
      \def\r{2}
      \def\R{8}
      \foreach \x in {1,2,...,8} {\coordinate (\x) at
        (90-45*\x:6cm);}
      \coordinate (9) at (45:\r);
      \coordinate (10) at (0:\r);
      \coordinate (11) at (135:\r);
      \coordinate (12) at (90:\r);
      \draw[GreenLine] (1)--(2)--(3)--(4)--(5)--(6)--(7)--(8)--cycle;
      \draw[GreenLine] (1)--(5) (2)--(10)--(6) (3)--(11)--(7) (4)--(12)--(8);
      \foreach \x in {1,2,...,12} {\fill (\x) circle (2mm);}
      \foreach \x in {9,10,11,12} {\fill[color=white] (\x) circle
        (1mm);}
      \newcommand{\placedots}{\foreach \x in {1,2,3,4} {\coordinate (d\x) at (315-90*\x:.7);}}
      \newcommand{\drawdots}{\foreach \x in {1,2,3,4} {\draw (d\x) node [smallDot] {};}}
      \foreach \x in {0,1,2,...,15} {\node (n\x) at (22.5*\x:\R) [Rect] {};}
      \begin{scope}[BluePoly]
        \begin{scope}[shift={(n0)}] \placedots \filldraw (d1)--(d2) (d2)--(d3)--(d4)--cycle; \drawdots \end{scope}
        \begin{scope}[shift={(n1)}] \placedots \draw (d1)--(d2)--(d3)--(d4); \drawdots \end{scope}
        \begin{scope}[shift={(n2)}] \placedots \filldraw (d1)--(d2)--(d3)--cycle (d3)--(d4); \drawdots \end{scope}
        \begin{scope}[shift={(n3)}] \placedots \draw (d2)--(d3)--(d4) (d1)--(d3); \drawdots \end{scope}
        \begin{scope}[shift={(n4)}] \placedots \filldraw (d3)--(d4)--(d1)--cycle (d2)--(d3); \drawdots \end{scope}
        \begin{scope}[shift={(n5)}] \placedots \draw (d2)--(d3)--(d4)--(d1); \drawdots \end{scope}
        \begin{scope}[shift={(n6)}] \placedots \filldraw (d2)--(d3)--(d4)--cycle (d1)--(d4); \drawdots \end{scope}
        \begin{scope}[shift={(n7)}] \placedots \draw (d3)--(d4)--(d2) (d1)--(d4); \drawdots \end{scope}
        \begin{scope}[shift={(n8)}] \placedots \filldraw (d1)--(d2)--(d4)--cycle (d3)--(d4); \drawdots \end{scope}
        \begin{scope}[shift={(n9)}] \placedots \draw (d3)--(d4)--(d1)--(d2); \drawdots \end{scope}
        \begin{scope}[shift={(n10)}] \placedots \filldraw (d1)--(d3)--(d4)--cycle (d1)--(d2); \drawdots \end{scope}
        \begin{scope}[shift={(n11)}] \placedots \draw (d2)--(d1)--(d3) (d1)--(d4); \drawdots \end{scope}
        \begin{scope}[shift={(n12)}] \placedots \filldraw (d1)--(d2)--(d3)--cycle (d1)--(d4); \drawdots \end{scope}
        \begin{scope}[shift={(n13)}] \placedots \draw (d4)--(d1)--(d2)--(d3); \drawdots \end{scope}
        \begin{scope}[shift={(n14)}] \placedots \filldraw (d1)--(d2)--(d4)--cycle (d2)--(d3); \drawdots \end{scope}
        \begin{scope}[shift={(n15)}] \placedots \draw (d1)--(d2)--(d3) (d2)--(d4); \drawdots \end{scope}
      \end{scope}
    \end{scope}
  \end{tikzpicture}
  \caption{A graph that can be viewed as either $\link(\ncparts_4)$ or
    $\complex(\nchtrees_4)$ depending on whether or not the white
    vertices are included.\label{fig:homeo4}}
\end{figure}
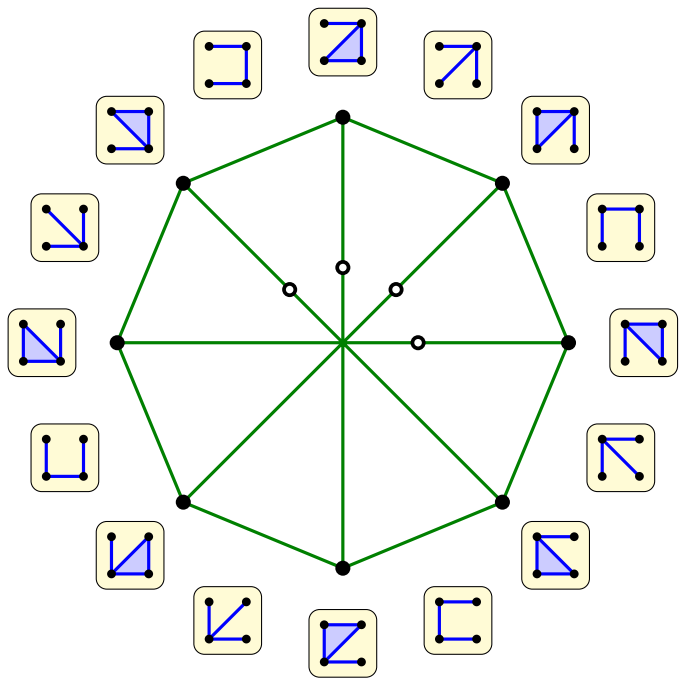

\begin{exmp}[Noncrossing hypertrees and partitions with $4$ vertices]
  The metric graph shown in Figure~\ref{fig:homeo4} can be viewed as
  either the noncrossing partition link $\ncparts_4$ when the $4$
  white vertices of degree $2$ are included, or as the noncrossing
  hypertree complex $\nchtrees_4$ when they are ignored.  In the
  noncrossing partition link there are $12$ vertices and $16$ edges
  all length $\frac{\pi}{3}$ and it is a subgraph of the spherical
  building shown in Figure~\ref{fig:bldg}.  When viewed as the
  noncrossing hypertree complex there are only $8$ vertices and $12$
  edges of varying lengths.  The $8$ edges around the outside have
  length $\frac{\pi}{3}$ and the four diagonal edges have length
  $\frac{2\pi}{3}$.  The noncrossing hypertree labels for the vertices
  and edges on the boundary have been drawn.  The diagonal edges are
  labeled by the four possible zig-zag trees.  The vertical diagonal,
  for example, is labeled by the tree that looks like the letter `Z'.
  The noncrossing partition labels correspond to weak proper orders on
  these hypertrees.  For every vertex and edge on the outside there is
  a unique such order and it is always a proper order.  The more
  interesting possibilities occur along the diagonals edges.  Let
  $\tau$ be the `Z' tree labeling the vertical diagonal.  The diagonal
  edge of the `Z' is below both other edges in its hyperedge poset and
  thus $\tau$ has two proper orderings depending on which of the other
  two edges is last, and this accounts for the two edges of the
  noncrossing partition link that are amalgamated to form the single
  edge in the noncrossing hypertree complex.  When a weakly proper
  order is chosen instead where both nondiagonal edges are last, the
  resulting simplex is the white vertex inbetween the two edges
  labeled by the proper orderings of $\tau$.
\end{exmp}

The previous sections have discussed two distinct ways of producing a
simplex in the noncrossing partition link $L$ from a noncrossing
hypertree.  The next remark compares them.

\begin{rem}[Two procedures]
  Definition~\ref{def:apartments-spheres} described how to turn a
  noncrossing hypertree $\tau$ with $k$ hyperedges into a simplicial
  sphere $\sphere(\tau)$ inside $L$ of dimension $k-2$ by considering
  all possible orderings of its hyperedges.  In particular, for each
  ordering, the hyperedges are added one at a time and the noncrossing
  partition of the noncrossing hyperforest formed by hyperedges added
  so far is recorded.  This essentially finds the join in the
  noncrossing partition lattice of the elements corresponding to the
  individual hyperedges.  The resulting chain in the noncrossing
  partition lattice describes a simplex in $L$.  The ordered
  noncrossing hypertree that produces this simplex can be considered a
  \emph{nonstandard name of the simplex} it produces.  It is merely a
  name because the process of converting ordering noncrossing
  hypertrees into simplices in $L$ is far from injective.  The
  \emph{standard name} for each simplex comes from
  Corollary~\ref{cor:simplices}.  It creates a weakly properly ordered
  noncrossing hypertree $\tau'$ from the labels of the adjacent
  elements in the chain that produces the simplex.
\end{rem}

One consequence of Lemma~\ref{lem:po-nc} is that when $\tau$ is a
properly ordered noncrossing hypertree, the two procedures agree.

\begin{lem}[Proper orderings and spheres]\label{lem:po-spheres}
  Let $\tau$ be a noncrossing hypertree with $k$ hyperedges.  For any
  proper ordering of its hyperedges, the simplex produced as part of
  the construction of $\sphere(\tau)$ is the same as the simplex whose
  standard label is this ordered noncrossing hypertree.  In
  particular, every simplex whose standard label is $\tau$ with some
  proper ordering of its hyperedges belongs to $\sphere(\tau)$.
\end{lem}

\begin{proof}
  Let $\sigma_1 \sigma_2 \cdots \sigma_k = c$ be the reduced
  factorization of the Coxeter element $c$ that corresponds to the
  chain in $\ncparts_{n+1}$ and the simplex in $L$ labeled by the
  properly ordered noncrossing hypertree $\tau$.  By
  Lemma~\ref{lem:po-nc} for each $j$ the product $\sigma_1 \cdots
  \sigma_j$ is equal to the noncrossing permutation of the noncrossing
  hyperforest formed by the first $j$ hyperedges of $\tau$, which
  means that both procedures produce the same chain and the same
  simplex.
\end{proof}

As a consequence of Lemma~\ref{lem:po-spheres}, the noncrossing
partition link is a union of its noncrossing apartments.  Now that the
properly orderings of a noncrossing hypertree $\tau$ live in a common
sphere, it is easy to analyze how that they fit together.  The first
step is to note which simplices form hemispheres.

\begin{lem}[Hemispheres]\label{lem:hemispheres}
  Let $\tau$ be a noncrossing hypertree and let $e$ and $e'$ be
  hyperedges of $\tau$.  The union of the simplices in $\sphere(\tau)$
  that are constructed from the orderings of $\tau$ in which $e$ comes
  before $e'$ form a hemisphere.  
\end{lem}

\begin{proof}
  This follows immediately from the description of the hemispheres in
  Definition~\ref{def:roots-hemis}.
\end{proof}

In fact, the proper orderings fit together to form a spherical
simplex.

\begin{lem}[Proper orderings and simplices]\label{lem:po-simplices}
  For every noncrossing hypertree $\tau$, the union of the closed
  simplices whose standard name is $\tau$ with a proper ordering of
  its hyperedges form a closed spherical simplex inside
  $\sphere(\tau)$.  As a consequence, the number of top-dimensional
  simplices in this union is equal to the number of linear extensions
  of an associated hyperedge poset $\poset(\tau)$.
\end{lem}

\begin{proof}
  Let $k$ be the number of hyperedges in $\tau$.  By
  Lemma~\ref{lem:po-spheres} every proper ordering of $\tau$ is the
  standard name of a simplex in $\sphere(\tau)$ and by
  Lemma~\ref{lem:hemispheres}, the $k-1$ covering relations in the
  hyperedge poset of $\tau$ determine $k-1$ hemispheres inside the
  $k-2$ dimensional sphere $\sphere(\tau)$.  A proper ordering on
  $\tau$ produces a simplex that lives in their intersection.  In
  fact, any simplex in the intersection of dimension $k-2$ must come
  from a proper ordering because its position implies that all
  covering relations are properly ordered.  Finally, the fact that the
  intersection of these hemispheres is a closed spherical simplex
  follows immediately from Proposition~\ref{prop:intersections}.  Note
  that the Hasse diagram of the hyperedge poset $\poset(\tau)$
  corresponds to the directed tree used in the proof of
  Proposition~\ref{prop:intersections}.
\end{proof}

For each noncrossing hypertree $\tau$ let $\simplex(\tau)$ be the
subcomplex of $L$ formed by the union of the closed simplices labeled
by the proper orderings of the hyperedges of $\tau$.  This is the
(subdivided) \emph{simplex of $\tau$} and the homeomoprhism between
$L$ and $C$ simply removes the subdivisions.  In the language of the
literature on factorizations into cycles, two simplices are merged if
and only if they correspond to \emph{equivalent} reduced
factorizations into cycles, the equivalence relation being the one
generated by permuting commuting adjacent cycles.

\begin{rem}[Interior simplices and facets]\label{rem:interior-facet}
  The standard names of the simplices of $L$ are not merely the proper
  orderings of noncrossing hypertrees, but weak proper orderings as
  well.  Let $\tau$ be a noncrossing hypertree.  Every weak proper
  ordering of the hyperedges of $\tau$ can be turned into a proper
  ordering by taking the noncrossing permutation that labels each
  adjacent pair in the corresponding chain in $\ncparts_{n+1}$ and
  factoring it into irreducbile permutations in some order.  In $L$
  this places the simplex corresponding to the weak proper ordering in
  the boundary of several simplices with proper orderings.  In
  particular, the weak proper ordering labels a simplex in the
  interior of $\simplex(\tau)$.  The simplices in the boundary of
  $\simplex(\tau)$, on the other hand, are labeled by (weak) proper
  orderings of noncrossing hypertrees obtained by merging hyperedges
  of $\tau$.
\end{rem}

\begin{figure}
  \begin{tikzpicture}[scale=.7]
    \begin{scope}[xshift=-4cm]
      \foreach \x in {1,2,...,5} {\coordinate (\x) at (90-72*\x:1.5cm);}
      \filldraw[YellowPoly] (0,0) circle (2.5cm);
      \draw[BlueLine] 
      (1) -- node[fill=yellow!20] {$d$} 
      (2)-- node[fill=yellow!20] {$b$} 
      (5)--  node[fill=yellow!20] {$c$} 
      (3)--  node[fill=yellow!20] {$a$} (4);
      \foreach \x in {1,2,...,5} {\fill (\x) circle (1mm);}
      \foreach \x in {1,2,...,5} {\draw (90-72*\x:2cm) node {\small \x};}
    \end{scope}
    \begin{scope}[xshift=2cm]
      \node (1) at (0,-1) [heposet] {$a$};
      \node (2) at (1,1) [heposet] {$c$};
      \node (3) at (2,-1) [heposet] {$b$};
      \node (4) at (3,1) [heposet] {$d$};
      \draw (1)--(2)--(3)--(4);
    \end{scope}
    \end{tikzpicture}
  \caption{A noncrossing tree $\tau$ with $5$ vertices and $4$ edges
    (left) and its hyperedge poset $\poset(\tau)$
    (right).\label{fig:union-A}}
\end{figure}
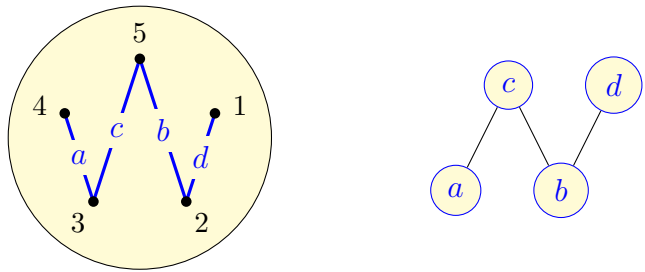

An example should help clarify these observations.

\begin{exmp}[Amalgamating simplices]
  Consider the noncrossing tree $\tau$ shown in
  Figure~\ref{fig:union-A} on the left.  It has $5$ vertices and $4$
  edges.  Its hyperedge poset, shown on the right, has the $4$ edges
  of $\tau$ as elements and $3$ covering relations.  There are exactly
  $5$ linear extensions of this hyperedge poset, or equivalently, $5$
  proper orderings of $\tau$.  These are $abcd$, $abdc$, $bacd$,
  $badc$ and $bdac$.  Each of these $5$ labels a maximal chain in
  $\ncparts_5$ and a spherical triangle in $L$.  Their union is shown
  in Figure~\ref{fig:union-B}.  The dotted lines and the vertex in the
  interior of this larger spherical triangle have standard labels that
  are weak proper orderings of $\tau$.  The sides of the triangle
  correspond to $\simplex(\tau')$ where $\tau'$ is one of the three
  noncrossing hypertrees formed by to merging two hyperedges of $\tau$
  as part of a covering relations in the noncrossing hypertree poset.
  The vertices arise from two such mergers.  The hypertrees that label
  these features are included in Figure~\ref{fig:union-B}.
\end{exmp}

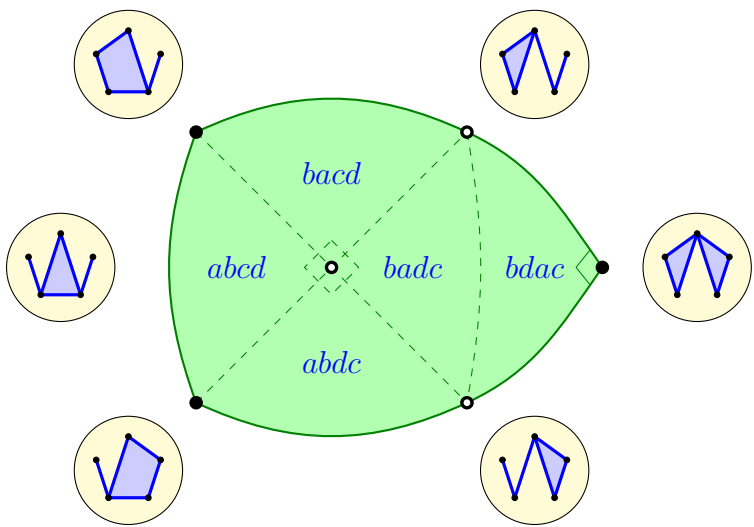
\begin{figure}
  \begin{tikzpicture}[scale=1]
    \begin{scope}[scale=1.8]
      \coordinate (1) at (-1,1);
      \coordinate (2) at (-1,-1);
      \coordinate (3) at (0,0);
      \coordinate (4) at (1,1);
      \coordinate (5) at (1,-1);
      \coordinate (6) at (2,0);
      \draw[GreenPoly] (6) [in=-25,out=125] to (4) 
      to [in=25,out=155] (1) to [out=-110,in=110] (2) 
      to [out=-25,in=-155] (5) [out=25,in=-125] to (6);
      \begin{scope}[GreenLine,dashed,thin]
        \def\r{.2}
        \draw (4) to [out=-80,in=80] (5); 
        \draw (1) to (3) to (5); 
        \draw (2) to (3) to (4); 
        \draw (0,\r)--(\r,0)--(0,-\r)--(-\r,0)--(0,\r);
      \end{scope}
      \def\r{.13}
      \def\R{2.04}
      \draw[GreenLine,thin]  (\R-\r,\r)-- (\R-1.8*\r,0) -- (\R-\r,-\r);
      \node at (-.7,0) [color=blue] {$abcd$};
      \node at (.6,0) [color=blue] {$badc$};
      \node at (0,.7) [color=blue] {$bacd$};
      \node at (0,-.7) [color=blue] {$abdc$};
      \node at (1.5,0) [color=blue] {$bdac$};
      \foreach \x in {1,2,...,6} {\fill (\x) circle (.5mm);}
      \foreach \x in {3,4,5} {\fill[color=white] (\x) circle (.25mm);}
      \newcommand{\placedots}{
        \filldraw[YellowPoly] (0,0) circle (.4cm);
        \foreach \x in {1,2,3,4,5} {\coordinate (d\x) at (90-72*\x:.25);}}
      \newcommand{\drawdots}{\foreach \x in {1,2,3,4,5} {\draw (d\x) node [smallDot] {};}}
      \def\r{1.5}
      \node (n1) at (\r,\r) {};
      \node (n2) at (\r,-\r) {};
      \node (n3) at (2.7,0) {};
      \node (n4) at (-\r,\r) {};
      \node (n5) at (-\r,-\r) {};
      \node (n6) at (-2,0) {};
      \begin{scope}[shift={(n1)}] \placedots 
        \filldraw[BluePoly] (d3)--(d4)--(d5)--cycle; 
        \draw[BlueLine]  (d1)--(d2)--(d5);  \drawdots 
      \end{scope}
      \begin{scope}[shift={(n2)}] \placedots 
        \filldraw[BluePoly] (d1)--(d2)--(d5)--cycle; 
        \draw[BlueLine]  (d4)--(d3)--(d5);  \drawdots 
      \end{scope}
      \begin{scope}[shift={(n3)}] \placedots 
        \filldraw[BluePoly] (d3)--(d4)--(d5)--cycle; 
        \draw[BluePoly]  (d1)--(d2)--(d5)--cycle;  \drawdots 
      \end{scope}
      \begin{scope}[shift={(n4)}] \placedots 
        \filldraw[BluePoly] (d2)--(d3)--(d4)--(d5)--cycle; 
        \draw[BlueLine]  (d1)--(d2);  \drawdots 
      \end{scope}
      \begin{scope}[shift={(n5)}] \placedots 
        \filldraw[BluePoly] (d1)--(d2)--(d3)--(d5)--cycle; 
        \draw[BlueLine]  (d3)--(d4);  \drawdots 
      \end{scope}
      \begin{scope}[shift={(n6)}] \placedots 
        \filldraw[BluePoly] (d2)--(d3)--(d5)--cycle; 
        \draw[BlueLine]  (d1)--(d2) (d3)--(d4);  \drawdots 
      \end{scope}
    \end{scope}
    \end{tikzpicture}
  \caption{The hyperedge poset of the tree $\tau$ shown in
    Figure~\ref{fig:union-A} has $5$ linear extensions, or
    equivalently there are $5$ proper orderings of $\tau$: $abcd$,
    $abdc$, $bacd$, $badc$ and $bdac$.  Each of these $5$ labels a
    spherical trianlge in the noncrossing partition link $\ncparts_5$
    with angles $\frac{\pi}{2}$, $\frac{\pi}{3}$ and $\frac{\pi}{3}$.
    These $5$ triangles fit together as shown to form a large
    spherical triangle with angle $\frac{2\pi}{3}$, $\frac{2\pi}{3}$
    and $\frac{\pi}{2}$.  The labels are the noncrossing hypertree
    labels of the $3$ vertices and $3$ sides of the large
    simplex.\label{fig:union-B}}
\end{figure}

\renewcommand{\themain}{\ref{main:geometry}}
\begin{main}[Topology and Geometry]
  The noncrossing hypertree complex is naturally homeomorphic to the
  noncrossing partition link. As a consequence, the pieceswise
  spherical metric on the latter induces a piecewise spherical metric
  on the former.
\end{main}

\begin{proof}
  The map between the noncrossing hypertree complex and the
  noncrossing partition link begins with the vertices
  (Definition~\ref{def:vertices}) and then it proceeds up through the
  skeleta.  When it is time to send the simplex of $C$ labeled by the
  noncrossing hypertree $\tau$ to $L$, the map on its boundary has
  already been defined and it is sent to the boundary of
  $\simplex(\tau)$ inside $L$ (Lemma~\ref{lem:po-simplices} and
  Remark~\ref{rem:interior-facet}).  Simply extend this map to the
  interior of the simplex labeled $\tau$ in $C$ to the interior of the
  simplex shaped subcomplex $\simplex(\tau)$ in $L$.  By construction
  the completed map is a continuous bijection and thus a
  homeomorphism.
\end{proof}

One way to restate this result is that the noncrossing hypertree
complex is a simplicial coarsening of the noncrossing partition link.
This is similar in spirit to the polyhedral coarsening discussed by
Nathan Reading in \cite{Re12}.  In order to distinguish between the
two simplicial structures, when both are under discussion, the
simplices in $L$ are callled \emph{partition simplices} and the
simplices in $C$ are called \emph{tree simplices}.  For example, if
$\tau$ is a noncrossing hypertree, then $\simplex(\tau)$ is a single
tree simplex but it is a union of partition simplices.  The relative
efficiency of $C$ as a simplicial structure over that of $L$ is clear
once we compare the number of vertices and chambers in each.

\begin{rem}[Relative sizes]
  The noncrossing partition link $L = \link(\ncparts_{n+1})$ has
  $\catalan_{n+1} = \frac{1}{n+2} \binom{2n+2}{n+1}$ many vertices and
  $(n+1)^{n-1}$ many partition chambers.  By contrast there are only
  $(n+1)(n-1)$ many vertices and $\catalan^{(2)}(A_n) = \frac{1}{2n+1}
  \binom{3n}{n}$ many tree chambers in the noncrossing hypertree
  complex $C = \complex(\nchtrees_{n+1})$.  In other words,
  exponentially many vertices and a superexponential number of
  chambers are amalgamated into a quadratic number of vertices and a
  mere exponential number of chambers.  See Table~\ref{tbl:ncht-ncpl}.
\end{rem}

\begin{table}
  \begin{center}
    \begin{tabular}{ccc}
      $\begin{array}{|r|r|r|r|}
        \hline
        n & d & \card{V} & \textrm{chambers}\\
        \hline
        \hline
        3 & 1 & 8 & 12\\
        4 & 2 & 15 & 55\\
        5 & 3 & 24 & 273\\
        6 & 4 & 35 & 1,428\\
        7 & 5 & 48 & 7,752\\
        8 & 6 & 63 & 43,263\\
        9 & 7 & 80 & 246,675\\
        \hline
      \end{array}$
      &
      \hspace*{1em}
      &
      $\begin{array}{|r|r|r|r|}
        \hline
        n & d & \card{V} & \textrm{chambers}\\
        \hline
        \hline
        3 & 1 & 12 & 16\\
        4 & 2 & 40 & 125\\
        5 & 3 & 130 & 1,296\\
        6 & 4 & 427 & 16,807\\
        7 & 5 & 1,428 & 262,144\\
        8 & 6 & 4,860 & 4,782,969\\
        9 & 7 & 16,794 & 100,000,000\\
        \hline
      \end{array}$
    \end{tabular}
  \end{center}
  \caption{The number of vertices and chambers for the noncrossing
    hypertree complex $\complex(\nchtrees_{n+1})$ (left) and the
    noncrossing partition link $\link(\ncparts_{n+1})$ (right) in low
    dimensions.\label{tbl:ncht-ncpl}}
\end{table}

\section{Theorem~\ref{main:duality}: Simplices and Spheres\label{sec:duality}}

This section completes the proof the Theorem~\ref{main:duality}, but
since each tree simplex is labeled by a noncrossing hypertree $\tau$
and a sphere containing it inside the noncrossing partition link has
already been identified for every noncrossing hypertree $\tau$, it
suffices to show that $\sphere(\tau)$ is a union of tree simplices and
thus a subcomplex in the new cell structure.  This follows from an
understanding of the process of standardizing the names of simplices
labeled by ordered hypertrees.

\begin{lem}[Standard names]\label{lem:standard-names}
  Let $\tau$ be an ordered hypertree with hyperedge permutations
  $\sigma_1, \sigma_2, \cdots, \sigma_k$ where the subscripts
  indicates the ordering.  If $\tau'$ is the standard name of the
  simplex corresponding to $\tau$ then $\tau'$ is the properly ordered
  noncrossing hypertree $\tau'$ with the same number of hyperedges
  permutations $\sigma'_1, \sigma'_2, \ldots, \sigma'_k$ and the same
  multiset of hyperedge sizes.  Moreover, for any permutation $\pi$ of
  $k$ elements such that $\sigma_{\pi(1)} \sigma_{\pi(2)} \cdots
  \sigma_{\pi(k)}$ is a proper ordering of the hyperedge permutations
  of $\tau$, the permutations $\sigma'_j$ in the standard name satisfy
  the equations $\sigma_j \beta_j = \beta_j \sigma'_j$ where $\beta_j$
  is the product of the permutations $\sigma_i$ with $i<j$ that occur
  to the right of $\sigma_j$ in the chosen proper ordering of $\tau$
  in the order they occur.  
\end{lem}

\begin{proof}
  Let $e_i$ be the hyperedge of $\tau$ corresponding to the hyperedge
  permutation $\sigma_i$ and for each $j$ define three permutations.
  Let $\tau_j$, $\alpha_j$ and $\beta_j$ be the noncrossing
  permutations of the noncrossing hyperforests whose hyperedges are
  (1) those $e_i$ with $i \leq j$, (2) those $e_i$ with $i<j$ and
  $\sigma_i$ to the left of $\sigma_j$ in the chosen proper ordering
  of $\tau$ and (3) those $e_i$ with $i<j$ and $\sigma_i$ to the right
  of $\sigma_j$ in the chosen proper ordering of $\tau$, respectively.
  Because the order is proper, the product of the $\sigma_i$'s in this
  order is a reduced factorization of the Coxeter element $c = \tau_k$
  (Lemma~\ref{lem:po-nc}).  Next, use Proposition~\ref{prop:rewriting}
  to find a new reduced factorization whose first $j$ factors, as a
  set, are the permutations $\sigma_i$ with $i \leq j$.  Since the
  order in which they occur is a proper ordering of the subhyperforest
  with only these $j$ hyperedges, the product of these $j$
  permutations in this order is $\tau_j$.  Similarly, the product of
  the permutations that occur before and after $\sigma_j$ in this
  ordering are $\alpha_j$ and $\beta_j$, respectively.  In other words
  $\tau_j = \alpha_j \sigma_j \beta_j$ and by
  Proposition~\ref{prop:rp} this is a reduced factorization of
  $\tau_j$.  If $\sigma'_j$ is defined as the unique permutation that
  solves the equation $\sigma_j \beta_j = \beta_j \sigma'_j$, then
  $\tau_j = \alpha_j \beta_j \sigma'_j$ and since $\alpha_j \beta_j$
  is the properly ordered product of all $\sigma_i$ with $i<j$, its is
  equal to $\tau_{j-1}$.  Thus $\tau_j = \tau_{j-1} \sigma'_j$.  By
  Proposition~\ref{prop:rp-chains} these labels $\sigma'_j$ are the
  unique labels on the chain $t = \tau_0 < \tau_1 < \cdots < \tau_k=c$
  and the corresponding hyperedges form a weakly properly ordered
  noncrossing hypertree $\tau'$.  But since conjugation does not
  change the cycle type of a permutation, $\sigma_j$ irreducible
  implies that $\sigma'_j$ is also irreducible.  As a consequence the
  ordering on the hyperedges on $\tau'$ is a proper order and not
  merely weakly proper order.
\end{proof}

Since the new hyperedge permutations $\sigma'_j$ are independent of
the chosen proper ordering of $\tau$, different proper orderings can
be chosen for each $j$ to minimize the number of conjugations
required.  For example it is sufficient to focus on inversions in the
hypertree poset.

\begin{defn}[Inversions in a partial order]
  Let $P$ be a finite poset with $k$ elements $e_1, e_2, \ldots, e_k$
  where the subscripts indictate a linear ordering of the elements of
  $P$.  If $i<j$ as integers but $e_i > e_j$ in the poset ordering of
  $P$ then \emph{$e_i$ inverts $e_j$}.
\end{defn}

\begin{rem}[Fewer conjugations]\label{rem:fewer-conj}
  Let $\tau$ be an ordered noncrossing hypertree with hyperedges $e_1,
  e_2, \ldots e_k$ and hyperedge permutations $\sigma_1, \sigma_2,
  \ldots \sigma_k$ where the subscripts indicate the linear order.
  For each $j$ there is a proper ordering of $\tau$ where the only
  hyperedges that occur after $e_j$ are those that are above $e_j$ in
  the hyperedge poset of $\tau$.  By choosing this proper order in
  Lemma~\ref{lem:standard-names}, it is clear that the hyperedge
  permutation $\sigma'_j$ in the standardized hypertree $\tau'$
  satisfies the equation $\sigma_j \beta_j = \beta_j \sigma'_j$ where
  $\beta_j$ is a product of the hyperedge permutations of the
  hyperedges that invert $e_j$ in $\poset(\tau)$ in any order that is
  proper for the noncrossing hyperforest that they form.
\end{rem}

The number of necessary conjugations can be further reduced but this
is sufficient for our purposes.  

\begin{exmp}[Standard names]\label{ex:standard}
  Let $e_1 = \{3,4,5\}$, $e_2 = \{5,6\}$, $e_3 = \{2,3\}$, $e_4 =
  \{1,3\}$ and $e_5 = \{5,7,8,9\}$ be a linear ordering of the $5$
  hyperedges of the noncrossing tree shown in on the righthand side of
  Figure~\ref{fig:ncht}.  As noted in Definition~\ref{def:edge-posets}
  the unique proper ordering of these $5$ hyperedges is $e_2 < e_5 <
  e_1 < e_4 < e_3$.  The hyperedge $e_3$ is not inverted by any other
  hyperedge, so $\sigma'_3 = \sigma_3 = (2,3)$.  On the other hand,
  $e_5$ is inverted by $e_1$, $e_4$ and $e_3$ so $\beta_5 = \sigma_1
  \sigma_4 \sigma_3 = (3,4,5)(1,3)(2,3) = (1,2,3,4,5)$ and $\sigma'_5$
  is $(1,7,8,9)$ because $(5,7,8,9)(1,2,3,4,5) = (1,2,3,4,5,7,8,9) =
  (1,2,3,4,5)(1,7,8,9)$.  In Figure~\ref{fig:standard} the original
  ordered noncrossing hypertree $\tau$ is shown on the left and the
  properly ordered noncrossing hypertree $\tau'$ to which it
  corresponds is shown on the right.
\end{exmp}

\begin{figure}
  \begin{tikzpicture}[scale=1]
    \begin{scope}[xshift=-3cm]
      \foreach \x in {1,2,...,9} {\coordinate (\x) at (90-40*\x:1.5cm);}
      \filldraw[YellowPoly] (0,0) circle (2.5cm);
      \draw[BlueLine] (2)--(3)--(1) (5)--(6);
      \draw[BluePoly] (3)--(4)--(5)--cycle;
      \draw[BluePoly] (5)--(7)--(8)--(9)--cycle;
      \draw[BlueLine] (0,-1.6) node {\small $1$};
      \draw[BlueLine] (-1,-1.3) node {\small $2$};
      \draw[BlueLine] (1.6,-.3) node {\small $3$};
      \draw[BlueLine] (.8,.3) node {\small $4$};
      \draw[BlueLine] (-.7,.3) node {\small $5$};
      \foreach \x in {1,2,...,9} {\fill (\x) circle (.7mm);}
      \foreach \x in {1,2,...,9} {\draw (90-40*\x:2cm) node {\small
          \x};}
    \end{scope}
    \begin{scope}[xshift=3cm]
      \foreach \x in {1,2,...,9} {\coordinate (\x) at (90-40*\x:1.5cm);}
      \filldraw[YellowPoly] (0,0) circle (2.5cm);
      \draw[BlueLine] (3) (1)--(2)--(3)--(6);
      \draw[BluePoly] (3)--(4)--(5)--cycle;
      \draw[BluePoly] (1)--(7)--(8)--(9)--cycle;
      \draw[BlueLine] (0,-1.6) node {\small $1$};
      \draw[BlueLine] (0,-.5) node {\small $2$};
      \draw[BlueLine] (1.6,-.3) node {\small $3$};
      \draw[BlueLine] (1,.6) node {\small $4$};
      \draw[BlueLine] (-.3,1) node {\small $5$};
      \foreach \x in {1,2,...,9} {\fill (\x) circle (.7mm);}
      \foreach \x in {1,2,...,9} {\draw (90-40*\x:2cm) node {\small \x};}
    \end{scope}
    \end{tikzpicture}
  \caption{The ordered noncrossing hypertree $\tau$ on the left
    corresponds to the properly ordered noncrossing hypertree $\tau'$
    on the right.\label{fig:standard}}
\end{figure}
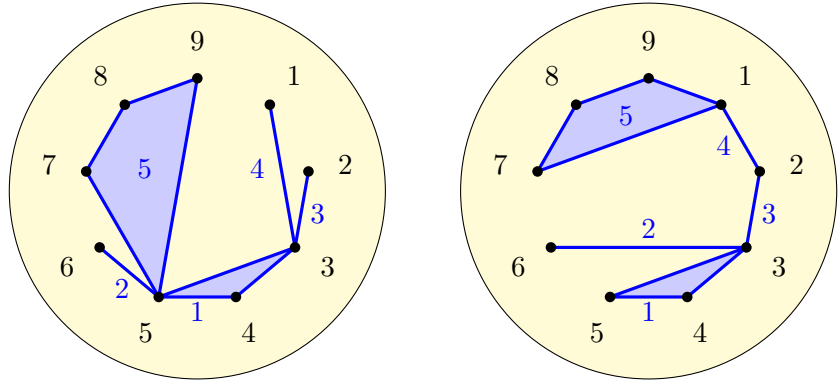

\begin{lem}[Facets and proper orderings]\label{lem:facets}
  If a simplex $S$ labeled by a properly ordered noncrossing hypertree
  $\tau$ has a facet $F$ labeled by a weak proper ordering of $\tau$,
  then this facet is a facet of exactly one other simplex $S'$ with a
  properly ordered noncrossing tree as its label.  Moreover, the label
  on $S'$ is another proper ordering of $\tau$ and the new ordering
  differs from the old by a single transposition of the order on a
  pair of hyperedges adjacent in the original ordering and
  incomparable in the hyperedge poset.
\end{lem}

\begin{proof}
  Recall that when passing to a facet a pair of adjacent hyperedge
  permutations are multipled.  In order for the result to be a weak
  proper ordering (as opposed to an extremely weak one), the
  corresponding hyperedges involved must be disjoint and the product
  of the two permutations is a noncrossing permutation with exactly
  two nontrivial blocks.  When undoing this process, one noncrossing
  permutation label is split/factored.  The only way that the factored
  labeling can correspond to a proper ordering of a hypertree is if
  the recently formed noncrossing permutation with two nontrivial
  blocks is the one that is factored and the factoring splits it into
  two irreducible permutations.  There are only two such factorings.
  One returns to the ordering of $\tau$ and the other to a different
  ordering of $\tau$.
\end{proof}

From this it quickly follows that the simplicial spheres already
identified in the noncrossing partition link remain simplicial spheres
in the new cell structure.

\begin{lem}[Union of tree simplices]\label{lem:union-tree-simplices}
  Let $\tau$ and $\tau'$ be noncrossing hypertrees with the same
  number of hyperedges.  If one proper ordering of $\tau'$ labels a
  partition simplex in $\sphere(\tau)$, then every proper ordering of
  $\tau'$ describes a partition simplex in $\sphere(\tau)$.  As a
  consequence, the simplicial sphere $\simplex(\tau)$ with its $k!$
  partition simplices can be viewed instead as a union of tree
  simplices.
\end{lem}

\begin{proof}
  Consider two top-dimensional partition simplices $S$ and $S'$ in
  $\simplex(\tau')$ that share a facet $F = S \cap S'$ and assume that
  $S$ is a partition simplex in $\sphere(\tau)$.  Note that $S$ is
  necessarily top-dimensional in $\sphere(\tau)$ because of the
  equality of the number of hyperedges.  Since spheres are manifolds,
  there is a unique simplex $S''$ in $\sphere(\tau)$ that shares the
  facet $F$ and with $S$, by Lemma~\ref{lem:standard-names} it is
  labeled by a properly ordered noncrossing hypertree and by
  Lemma~\ref{lem:facets} $S'$ and $S''$ must be one and the same.
  Finally, since the union of the simplices labeled by all of the
  proper orderings of $\tau'$ form a spherical simplex, this type of
  facet-sharing adjacency can be used to conclude that all of the
  partition simplices in $\simplex(\tau')$ lie in $\sphere(\tau)$.
\end{proof}

Theorem~\ref{main:duality} is now immediate.

\renewcommand{\themain}{\ref{main:duality}}
\begin{main}[Simplices and Spheres]
  Every spherical simplex in the metric noncrossing hypertree complex
  of any dimension is contained in a subcomplex isometric to a unit
  sphere of the same dimension.  In fact, in the top dimension (and
  conjecturally in all dimensions) there is a natural map from
  simplices to spherical subcomplexes that establishes a bijection
  between these two sets.
\end{main}

\begin{proof}
  The first assertion is immediate since each simplex in the
  noncrossing hypertree complex is labeled by a noncrossing hypertree
  $\tau$ and $\sphere(\tau)$ is a sphere inside the noncrossing
  partition link of the same dimension containing the simplex labeled
  $\tau$ which, by Lemma~\ref{lem:union-tree-simplices}, remains a
  simplicial sphere in the simplified cell structure.  Finally, by
  Proposition~\ref{prop:nc-apartments} the apartments labeled by
  noncrossing trees are the only metric spheres of this dimension
  inside the noncrossoing partition link and therefore the only metric
  spheres of this dimension inside the noncrossing hypertree complex.
\end{proof}

The bijection established above between spheres and simplices in the
top dimension (i.e. between tree simplices and apartments) is one that
I firmly believe extends to all dimensions, but I do not currently
have a proof of this conjecture.  To do so would involve two steps.
The first relatively easy step is showing that the map from simplices
to spheres sending $\simplex(\tau)$ to $\sphere(\tau)$ for each
noncrossing hypertree $\tau$ is injective.  The second more difficult
step would be to show that no other metric spheres exist as
subcomplexes of the metric noncrossing hypertree complex.  Both steps
are true for the small values of $n$ where explicit checks are
feasible.  Since it has not been conclusively proved that these are
the only spheres in the noncrossing hypertree complex, the known
spheres deserve a name.  For the remainder of the article, the spheres
of the form $\sphere(\tau)$ for some noncrossing hypertree $\tau$ are
called \emph{special spheres}.

\section{Theorem~\ref{main:bijections}: Bijections}

In this section Theorem~\ref{main:bijections} is proved using some
elementary automorphisms of the noncrossing hypertree complex.

\begin{rem}[Automorphisms]
  Let $\dih(k)$ denote the \emph{dihedral group} of order $2k$ that
  arises as the isometry group of a regular $k$-gon.  The
  \emph{natural automorphisms} of the poset of noncrossing hypertrees
  $\nchtrees_k$ and the noncrossing partition lattice $\ncparts_k$
  (and the corresponding simplicial automorphisms of the noncrossing
  hypertree complex $\complex(\nchtrees_k)$ and the noncrossing
  partition link $\link(\ncparts_k$) are the automorphisms that arise
  from the dihedral group $\dih(k)$ of isometries of the underlying
  $k$-gon in the plane used in their definition.  These form half of
  the full automorphism group.  The other half become visible when
  (the dual of) the noncrossing hypertree poset is identified with
  poset of dissections of a polygon with twice as many sides into
  even-sided subpolygons.  The obvious automorphisms in this context
  form a dihedral group $\dih(2k)$ with twice as many elements with
  the previous automorphisms being those that preserve the bipartite
  coloring of the vertices as black and white.  A representative new
  automorphism is the rotation by $\frac{2\pi}{2k} = \frac{\pi}{k}$.
  In the context of the noncrossing partition lattice this becomes the
  well-known Kreweras map sending each noncrossing permutation to its
  right complement.  In the context of the noncrossing hypertree
  poset, it sends each noncrossing hypertree $\tau$ to its
  ``opposite'' defined by picking a proper ordrering of $\tau$,
  reversing it, and then finding the properly ordered hypertree
  $\tau'$ that corresponds to this ordering of $\tau$.  As in the
  noncrossing partition context, this map is not an involution.  Doing
  it twice produces a hypertree $\tau''$ that is a $\frac{2\pi}{k}$
  rotation of $\tau$ as is easy to see from the polygon dissection
  viewpoint.  This is a map with many interesting properties.  They
  are not explored here but are worthy of further study.
\end{rem}

The only automorphism used below is the natural one induced by a
reflection of the underlying polygon.

\begin{lem}[Reflection of spheres]\label{lem:reflect-sphere}
  If $\refl$ is an involution on the poset of noncrossing hypertrees
  induced by a fixed reflection of a polygon in the plane used to
  define them, then for any noncrossing hypertree $\tau$,
  $\sphere(\tau)$ and $\sphere(\refl(\tau))$ are isomorphic as
  simplicial spheres and there is a natural bijection between their
  top dimensional tree simplices.
\end{lem}

\begin{proof}
  This is nearly immediate since the reflection on the polygon induces
  an automorphism of the noncrossing partition lattice that sends the
  special Boolean subposet that corresponds to $\sphere(\tau)$ to the
  special Boolean subposet that corresponds to $\sphere(\refl(\tau))$.
  This is because the reflection of the noncrossing partition
  constructed from the connected components of the noncrossing
  hyperforest formed by a subset of hyperedges of $\tau$ is the the
  noncrossing partition constructed from the connected components of
  the noncrossing hyperforest formed by the corresponding subset of
  the reflected hyperedges of $\tau$.
\end{proof}

Reflections also produce a kind of duality between simplices and
special spheres.

\begin{lem}[Reflections and inclusions]\label{lem:reflect-include}
  Let $\refl$ be an involution on the set of noncrossing hypertrees
  induced by a fixed reflection of a polygon in the plane used to
  define them.  If $\tau$ and $\tau'$ are noncrossing hypertrees with
  the same number of hyperedges such that the simplex labeled $\tau'$
  belongs to the sphere labeled $\tau$, then the simplex labeled
  $\refl(\tau)$ belongs the sphere labeled $\refl(\tau')$.  In
  particular, for any fixed noncrossing hypertree $\tau$ there is a
  bijection between the tree simplices of top dimension inside
  $\sphere(\refl(\tau))$ and the special spheres containing
  $\simplex(\tau)$ as a simplex of top-dimension.
\end{lem}

\begin{proof}
  When working with the reflected versions it is convenient to relabel
  the vertices of the underlying polygon so that the new vertex $i$ is
  the image of $i$ under the fixed reflection of the polygon. This
  makes the two situations easier to compare.  The effect that this
  has is that the vertices are proceeding in the opposite direction
  around the boundary, the new Coxeter element is $c^{-1}$, where $c$
  is the old Coxeter element.  Similarly, the hyperedge permutations
  of $\refl(\tau)$ are the inverses of the hyperedge permutations of
  $\tau$ and the poset $\poset(\refl(\tau))$ is the dual of
  $\poset(\tau)$.  More precisely, because of the vertex relabeling,
  the hyperedges of $\tau$ and $\refl(\tau)$ are identical, but the
  local linear orderings are reversed.  Moreover, if $\sigma_1
  \sigma_2 \cdots \sigma_k$ is a reduced factorization of $c$ coming
  from a properly ordered noncrossing hypertree $\tau$, then
  $\sigma_k^{-1} \cdots \sigma_2^{-1} \sigma_1^{-1}$ is the reduced
  factorization of $c^{-1}$ that coresponds to a proper ordering of
  the noncrossing hypertree $\refl(\tau)$ in the relabeled polygon.
  With this as background, the main idea of the proof is that when the
  simplex labeled $\tau'$ belongs to the sphere labeled $\tau$ then
  there are two orderings of $\tau$, one proper and one improper so
  that when Lemma~\ref{lem:standard-names} is applied to the improper
  ordering on $\tau$ it leads to a proper ordering of $\tau'$.  The
  two orderings together define a permutation $\pi$ of $k$ elements,
  where $k$ is the common number of hyperedges in $\tau$ and $\tau'$.
  The subscripts of the permutations in the corresponding reduced
  factorization of $c^{-1}$ coming from this corresponding proper
  ordering of $\refl(\tau')$ can be permuted using the permutation
  $\pi$ and the global order-reversing permutation so that the process
  of standardizing this improper ordering of $\refl(\tau')$ undoes
  precisely those conjugations that produced $\tau'$ from $\tau$.
  Concretely relabel $(\sigma'_i)^{-1}$ as $\rho_j$ with $j =
  \rev(\pi(i))$ where $\rev$ is the function that globally reverses
  the ordering on the numbers $1$ through $k$.  The standard name that
  results is image of the standard name of $\tau$ in the original
  proper ordering under the map $\refl$.
\end{proof}

The following is an extended concrete example that illustrates how
this relabeling causes the conjugations to be undone.

\begin{exmp}[Reversing direction]\label{ex:reverse}
  This example continues working with the ordered hypertree $\tau$
  given in Example~\ref{ex:standard} and properly ordered hypertree
  $\tau'$ that labels the simplex to which $\tau$ corresponds.  All of
  the notation established there for the hyperedge permutations is
  carried over here.  In particular, note that this is an instance
  where the tree simplex labeled by $\tau'$ belongs to the sphere
  labeled by $\tau$.  One advantage is that both $\tau$ and $\tau'$,
  viewed as unordered hypertrees, have a unique proper ordering,
  thereby eliminating a potential source of confusion.

  The factorization of $c$ corresponding to the unique proper ordering
  of the hyperedges of $\tau$ is \[\sigma_2 \sigma_5 \sigma_1 \sigma_4
  \sigma_3 = (5,6) (5,7,8,9) (3,4,5) (1,3) (2,3).\] Thus the
  permutation $\pi$ used in Lemma~\ref{lem:standard-names} has
  $\pi(1)=2$, $\pi(2)=5$, $\pi(3)=1$, $\pi(4)=4$ and $\pi(5)=3$.  The
  standardization process leads to the reduced factorization
  \[\sigma'_1 \sigma'_2 \sigma'_3 \sigma'_4\sigma'_5 = (3,4,5) (3,6)
  (2,3) (1,2) (1,7,8,9)\] of $c$ corresponding to the unique proper
  ordering of $\tau'$.  See Figure~\ref{fig:standard}.  For the
  reflected hypertree $\refl(\tau')$, with its relabeled vertices, the
  reduced factorization of $c^{-1}$ is $(\sigma'_5)^{-1}
  (\sigma'_4)^{-1} (\sigma'_3)^{-1} (\sigma'_2)^{-1}
  (\sigma'_1)^{-1}$.  Concretely this is the reduced product
  \[ (9,8,7,1) (2,1) (3,2) (6,3) (5,4,3) = (9,8,7,6,5,4,3,2,1) \]
  
  Relabel these $5$ permutations as follows.  Let $\rho_4 =
  (\sigma'_5)^{-1} = (9,8,7,1)$, $\rho_2 = (\sigma'_4)^{-1} = (2,1)$,
  $\rho_1 = (\sigma'_3)^{-1} = (3,2)$, $\rho_5 = (\sigma'_2)^{-1} =
  (6,3)$ and $\rho_3 = (\sigma'_1)^{-1} = (5,4,3)$.  The subscripts
  $\rho_j = (\sigma'_i)^{-1}$ are chosen so that $j = \rev(\pi(i))$
  where $\rev$ is the function that globally reverses the ordering on
  the numbers $1$ through $k$, with $k=5$ in this example.  In this
  case it switches $1$ and $5$, switches $2$ and $4$ and fixes $3$.
  For example, $\pi(5) = 2$ and $\rev(2) = 4$ so $(\sigma'_5)^{-1}$ is
  relabeled $\rho_4$.  See Figure~\ref{fig:reverse}.

  When Lemma~\ref{lem:standard-names} is applied to this ordering
  $\rho_4 \rho_2 \rho_1 \rho_5 \rho_3$ of the hyperedges of
  $\refl(\tau')$, $\rho_5 = (6,3)$ is inverted by $\rho_3 = (5,4,3)$
  so $\rho'_5 = (6,5)$ because $(6,3)(5,4,3) = (6,5,4,3) =
  (5,4,3)(6,5)$, but $\rho_3$ and $\rho_1$ are not inverted at all so
  that $\rho'_3 = \rho_3 = (5,4,3)$ and $\rho'_1 = \rho_1 = (3,2)$.
  The permutation $\rho_2$ is inverted by $\rho_1$ so $\rho'_2 =
  (3,1)$, because $(2,1)(3,2) = (3,2,1) = (3,2)(3,1)$.  Finally,
  $\rho_4$ is inverted by $\rho_2$, $\rho_1$ and $\rho_3$.  Thus
  $\rho'_4 = (9,8,7,6)$ because $\rho_2 \rho_1 \rho_3 =
  (2,1)(3,2)(5,4,3) = (5,4,3,2,1)$ and $(9,8,7,1)(5,4,3,2,1) =
  (9,8,7,5,4,3,2,1) = (5,4,3,2,1)(9,8,7,5)$.  A close examination of
  these steps shows that they are undoing exactly the conjugations
  that were done in Example~\ref{ex:standard}.  The end result of this
  standardization is \[\rho'_1 \rho'_2 \rho'_3 \rho'_4 \rho'_5 = (3,2)
  (3,1) (5,4,3) (9,8,7,5) (6,5)\] and this is the reduced
  factorization of $c^{-1}$ corresponding to the unique proper
  ordering of the hyperedges of $\refl(\tau)$.
\end{exmp}

\begin{figure}
  \begin{tikzpicture}[scale=1]
    \begin{scope}[xshift=-3cm]
      \foreach \x in {1,2,...,9} {\coordinate (\x) at (90+40*\x:1.5cm);}
      \filldraw[YellowPoly] (0,0) circle (2.5cm);
      \draw[BlueLine] (3) (1)--(2)--(3)--(6);
      \draw[BluePoly] (3)--(4)--(5)--cycle;
      \draw[BluePoly] (1)--(7)--(8)--(9)--cycle;
      \draw[BlueLine] (-1.6,-.3) node {\small $1$};
      \draw[BlueLine] (-1,.6) node {\small $2$};
      \draw[BlueLine] (0,-1.6) node {\small $3$};
      \draw[BlueLine] (.3,1) node {\small $4$};
      \draw[BlueLine] (0,-.5) node {\small $5$};
      \foreach \x in {1,2,...,9} {\fill (\x) circle (.7mm);}
      \foreach \x in {1,2,...,9} {\draw (90+40*\x:2cm) node {\small \x};}
    \end{scope}
    \begin{scope}[xshift=3cm]
      \foreach \x in {1,2,...,9} {\coordinate (\x) at (90+40*\x:1.5cm);}
      \filldraw[YellowPoly] (0,0) circle (2.5cm);
      \draw[BlueLine] (2)--(3)--(1) (5)--(6);
      \draw[BluePoly] (3)--(4)--(5)--cycle;
      \draw[BluePoly] (5)--(7)--(8)--(9)--cycle;
      \draw[BlueLine] (-1.6,-.3) node {\small $1$};
      \draw[BlueLine] (-.8,.3) node {\small $2$};
      \draw[BlueLine] (0,-1.6) node {\small $3$};
      \draw[BlueLine] (.7,.3) node {\small $4$};
      \draw[BlueLine] (1,-1.3) node {\small $5$};
      \foreach \x in {1,2,...,9} {\fill (\x) circle (.7mm);}
      \foreach \x in {1,2,...,9} {\draw (90+40*\x:2cm) node {\small
          \x};}
    \end{scope}
    \end{tikzpicture}
  \caption{The ordered noncrossing hypertree $\refl(\tau')$ on the
    left that corresponds to the properly ordered noncrossing
    hypertree $\refl(\tau)$ on the right.\label{fig:reverse}}
\end{figure}
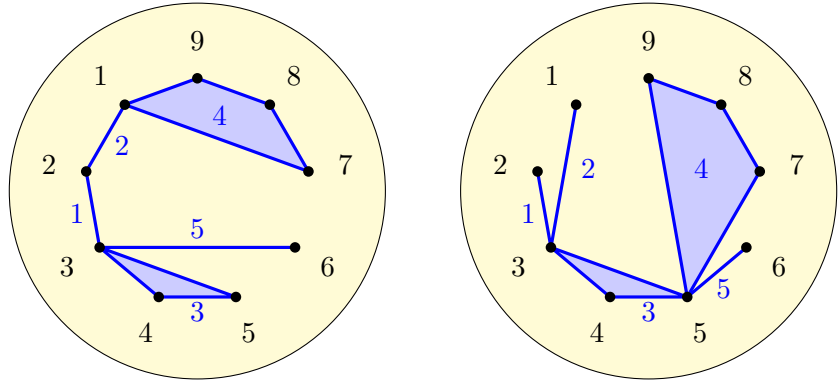

These lemmas combine to prove the theorem.

\renewcommand{\themain}{\ref{main:bijections}}
\begin{main}[Bijections]
  For every noncrossing hypertree $\tau$ there is a bijection between
  the number of special spheres containing the tree simplex labeled
  $\tau$ as a top-dimensional simplex and the set of tree simplices
  contained in the special sphere labeled $\tau$.  When $\tau$ is a
  noncrossing tree, this means that there is a bijection between the
  set $\{ \sigma \mid \chamber(\tau) \in \apartment(\sigma)\}$ of
  apartments containing the tree chamber labeled $\tau$ and the set
  $\{ \sigma \mid \chamber(\sigma) \in \apartment(\tau)\}$ of tree
  chambers in the apartment labeled~$\tau$.
\end{main}

\begin{proof}
  Let $\refl$ the involution on the set of noncrossing hypertrees
  induced by a fixed reflection of a polygon in the plane used to
  define them.  Let $A$, $C$ and $C'$ be the set of special spheres
  containing $\simplex(\tau)$ as a top-dimensional simplex, the set of
  top-dimensional simplices inside $\sphere(\tau)$ and the set of
  top-dimensional simplices inside $\sphere(\refl(\tau))$.
  Lemma~\ref{lem:reflect-include} gives a bijection between $A$ and
  $C'$ and Lemma~\ref{lem:reflect-sphere} gives a bijection between
  $C'$ and~$C$.  The final assertion is merely the special case where
  the simplices and spheres under discussion have the largest possible
  dimension.
\end{proof}

\section{Theorem~\ref{main:associahedra}: Associahedra}

This section completes the proof of the final main theorem,
Theorem~\ref{main:associahedra}, as part of an investigation of the
simplicial structure of special sphere and especially apartments.  The
special spheres with the least number of tree simplices turn out to be
cross polytopes and the apartments that seem to have the most turn out
to be simplicial associahedra.  The cross polytope case is considered
first since it is much easier to establish.  It is also less
surprising because of the following observation.

\begin{lem}[Subspheres]\label{lem:subspheres}
  If $\tau$ and $\tau'$ are noncrossing hypertrees with $k$ and $k-1$
  hyperedges, respectively, such that $\tau > \tau'$ is a covering
  relation in the noncrossing hypertree poset, then the simplicial
  sphere $\sphere(\tau)$ contains $\sphere(\tau')$ as a subcomplex and
  this equitorial subsphere divides $\sphere(\tau)$ into two equal
  hemispheres.  As a consequence, the $k-1$ facets of $\simplex(\tau)$
  divide $\sphere(\tau)$, isometric to $\sph^{k-2}$, into at least
  $2^{k-1}$ top dimensional tree simplices.
\end{lem}

\begin{proof}
  This is nearly immediate from the various constructions.  If $e$ and
  $e'$ are the hyperedges labeling the endpoints of the covering
  relation in $\poset(\tau)$ that are merged to form $\tau'$, then $e$
  and $e'$ share a vertex $v$.  Any ordering of the hyperedges of
  $\tau'$ can be turned into a ordering of $\tau$ in two different
  minimal ways by splitting the hyperedge back into $e$ and $e'$ and
  then breaking the tie between them by placing one before the other.
  If the ordering is thought of as a map to the reals, with the merged
  hyperedges being sent to a real number $r$, then breaking the tie
  means sending $e$ and $e'$ to $r+\epsilon$ and $r-\epsilon$ or to
  $r-\epsilon$ and $r + \epsilon$.  The tree simplex in
  $\sphere(\tau')$ labeled by this ordering of $\tau'$ is the facet
  between the two tree simplices in $\sphere(\tau)$ labeled by the two
  associated orderings of~$\tau$.
\end{proof}

Let $\tau$ be a noncrossing hypertree with $k$ hyperedges and note
that the simplicial structure determined by the intersections of the
$k-1$ subspheres of $\sphere(\tau)$ coming from the facets of $\tau$
is already that of a cross-polytope.  This arrangement is called a
\emph{metric cross polytope} because it has additional structure: the
simplicial structure is that of a cross polytope but also every
simplex lives in a metric subsphere.  The following is an easy
consequence of Lemma~\ref{lem:subspheres}.

\begin{lem}[Cross-polytopes]\label{lem:cross-poly}
  For each noncrossing hypertree $\tau$ with $k$ hyperedges, the
  special sphere $\sphere(\tau)$ has at least $2^{k-1}$
  top-dimensional tree simplices and when this minimum value is
  achieved, the structure of $\sphere(\tau)$ is that of a metric
  cross-polytope.
\end{lem}

\begin{proof}
  As noted above, the simplicial structure determined by the
  intersections of the $k-1$ subspheres of $\sphere(\tau)$ coming from
  the facets of $\tau$ is already that of a metric cross-polytope,
  dividing the sphere into $2^{k-1}$ top-dimensional simplices.  Since
  every top-dimensional tree simplex must live in one of these pieces,
  the inequality and the consequence of equality are immediate.
\end{proof}

Many noncrossing hypertrees achieve this theoretical minimum.

\begin{thm}[Cross-polytopes]\label{thm:cross-poly}
  If $\tau$ is noncrossing hypertree with $k$ hyperedges such that
  every hyperedge is either a minimum element or a maximum element in
  its hyperedge poset, then $\sphere(\tau)$ contains exactly $2^{k-1}$
  top-dimensional tree simplices and, as a consequence,
  $\sphere(\tau)$ is a metric cross-polytope.
\end{thm}

\begin{proof}
  Since by Lemma~\ref{lem:cross-poly} there are at least $2^{k-1}$
  top-dimensional tree simplices in $\sphere(\tau)$, it suffices to
  show that under these conditions $2^{k-1}$ is also an upper bound.
  If $e$ is a hyperedge of $\tau$ that labels a maximum element in
  $\poset(\tau)$, then by Remark~\ref{rem:fewer-conj}, it remains
  unchanged by the standardization process under any ordering of the
  hyperedges of $\tau$.  When $e$ is a hyperedge that labels a minimum
  element in $\poset(\tau)$ that $e$ has degree $d$ in the tree that
  is the Hasse diagram of this poset, then by
  Remark~\ref{rem:fewer-conj} there are at most $2^d$ possibilities
  for the new hyperedge permutation $\sigma'$ derived from $\sigma$ in
  the new tree created by the standardizing process.  Since the sum of
  the degrees of the minimal elements in $\poset(\tau)$ is equal to
  $k-1$, the number of covering relations it has, the total number of
  possible noncrossing hypertrees in $\sphere(\tau)$ that result from
  the standardizing process is at most $2^{k-1}$, the product of these
  choices over the various minimal elements.  The final assertion
  follows from Lemma~\ref{lem:cross-poly}.
\end{proof}

The noncrossing trees that satisfy this condition produce apartments
that are cross-polytopes.

\begin{rem}[Zig-zag trees]\label{rem:zig-zag}
  If $\tau$ is a noncrossing tree satisfying the hypothesis of
  Theorem~\ref{thm:cross-poly}, then $\tau$ has the structure of a
  single path that zig-zags back and forth across the polygon so that
  every edge has one of two possible slopes (assuming that the
  underlying convex polygon is regular).  The edges with one slope are
  the maximal elements in $\poset(\tau)$ and the edges with the other
  slope are the minimal elements in $\poset(\tau)$.  In addition to
  $\apartment(\tau)$ having the fewest possible number of tree
  chambers, the tree chamber $\chamber(\tau)$ contains the maximum
  number of partition chambers.  These arise from the orderings known
  as zig-zag permutations and they correspond to linear extensions of
  the zig-zag poset that is $\poset(\tau)$.  The number of these
  linear extensions is given by the sequence \texttt{oeis:A000111}
  which starts $1,1,2,5,16,61,272,1385,7936$, a set of numbers that
  also occur in the exponential generating function for $\sec(x)+
  \tan(x)$.  For further information about zig-zag posets and zig-zag
  permutations see \cite{EC1} or the references in listed in the
  Online Encyclopedia of Integer Sequences \cite{oeis}.  In a context
  very closely related to the results presented here, these numbers
  also appear in an article by K. Saito~\cite{Sa07}.
\end{rem}

The remainder of the section focuses on certain apartments in the
noncrossing hypertree complex that have the simplicial structure of a
simplicial associahedron.  That there are at least some simplicial
associahedra in the noncrossing hypertree complex is clear because of
its identification as a generalized cluster complex of type $A$
(Remark~\ref{rem:gen-cluster}).

\begin{rem}[Associahedra and generalized cluster complexes]
  In their foundational article on generalized cluster complexes
  \cite{FoRe05}, Fomin and Reading prove the generalized cluster
  complexes are nested in the following sense.  If $\Phi$ is a root
  system and $m \geq m'$ are two positive integers then, in their
  notation, the generalized cluster complex $\Delta^m(\Phi)$ contains
  the generalized cluster complex $\Delta^{m'}(\Phi)$ as a full
  subcomplex determined by a subset of vertex set.  In the case where
  $m=2$, $m'=1$ and $\Phi$ is a type $A$ root system, this means that
  the noncrossing hypertree complex contains a simplicial
  associahedron as a subcomplex.
\end{rem}

The noncrossing trees that label apartments with a simplicial
associahedral structure are those whose simplex is as small as
possible.  If the tree chamber labeled by a noncrossing tree $\tau$
contains only a single partition chamber, it is because the hyperedge
poset $\poset(\tau)$ has a unique linear extension, which in turn
means that $\poset(\tau)$ is already a linear ordering.  Two extreme
cases with this property are \emph{stars} where all edges in $\tau$
share a common vertex, and \emph{border trees} that consist of all but
one edge of the boundary cycle of the underlying polygon.  The full
set of trees with a linear hyperedge poset can be characterized as
those whose structure is that of a caterpillar.

\begin{defn}[Caterpillars]
  A noncrossing tree $\tau$ is called a \emph{caterpillar} if the set
  of edges in $\tau$ that live in the boundary of the polygon form a
  connected subtree. The name comes from viewing the edges in the
  boundary path as its \emph{backbone} and the other edges are its
  \emph{legs}.  Stars and border trees are the extreme cases where the
  caterpillar has as many or as few legs as possible.  For every
  caterpillar and for every $i$, the noncrossing partition of the
  hyperforest formed by the first $i$ edges of $\tau$ in the unique
  proper ordering of $\tau$, has a single nontrivial block.  In
  addition, these blocks are nested so that caterpillars correspond to
  the maximal chains in image of the annular poset formed by the basic
  hypertrees (see the example shown in
  Figure~\ref{fig:two-hyperedges}) when it is interpreted as a
  subposet of the noncrossing partition lattice using only the lower
  hyperedge of each basic hypertree.  The number of such maximal
  chains, and thus the number of caterpillars is $(n+1) 2^{n-1}$.  To
  see this pick a starting point in the bottom level and then move to
  the left or to the right at successive step.
\end{defn}

In this language Theorem~\ref{main:associahedra} asserts that the
apartment of a caterpillar is a simplicial associahdron.

\renewcommand{\themain}{\ref{main:associahedra}}
\begin{main}[Associahedra]
  Let $\tau$ be a noncrossing tree.  If the tree chamber labeled
  $\tau$ consists of a single partition chamber, then the apartment
  labeled $\tau$ is a simplicial associahedron.  In addition, the
  variety of simplicial associahedra produced in this way include all
  of the simplicial associahedra that are normal fans to the type $A$
  simple associahedra constructed by Hohlweg and Lange.
\end{main}

\begin{proof}[Proof sketch]\footnote{In a certain sense, the proof of
    this theorem is a matter of chasing definitions through the
    literature and the portion currently included here is merely a
    sketch of the main ideas.  The final version of this article will
    contain more details.}  Let $\tau$ be a catepillar and number the
  vertices of the polygon as follows.  Label the endpoint of the
  unique minimal edge in $\poset(\tau)$ that is a leaf in $\tau$ with
  the number $0$, and then the label the unique vertex that belongs to
  the nontrivial block at stage $i$ but not at stage $i-1$ with the
  number $i$.  This numbering corresponds to the vertex numbering used
  by Hohlweg and Lange in \cite{HoLa07} to construct their various
  realizations of the classical type $A$ simple associahedron.  The
  normal fans of these simple associahedra are the $c$-Cambrian fans
  of Reading, which have the desired structure of a simplicial
  associahedron \cite{Re06}.  The identifications of partition
  chambers made by Reading ultimately correspond to the
  identifications of partitions chambers into tree chambers made here.
\end{proof}

Now that these foundational results are in place, there are many
avenues for future research that might be pursued.  One obvious route
is to try and extend these results to the other types of finite
Coxeter groups, and another would to be re-examine earlier work
relating noncrossing partitions and associahedra to see whether the
arguments can be recast as processess that take place completely
within the noncrossing hypertree complex and/or the noncrossing
partition link.  Both the bijective maps between associahedra and
noncrossing partitions and the type-free proofs of the lattice
property might benefit from this type of re-examination.


\medskip

\noindent \textbf{Acknowledgements:} I would like to thank Professor
Kai-Uwe Bux for inviting me to spend two weeks visiting the University
of Bielefeld in early July 2015.  The stimulating conversations that I
had with him and Stefan Witzel while I was there laid the foundations
for the work presented here.  I would also like to thank Frederick
Chapoton, Michael Dougherty, Nathan Reading and Nathan Williams for
feedback on a version of this article drafted in late 2015.

\def\cprime{$'$}
\providecommand{\bysame}{\leavevmode\hbox to3em{\hrulefill}\thinspace}
\providecommand{\MR}{\relax\ifhmode\unskip\space\fi MR }
\providecommand{\MRhref}[2]{%
  \href{http://www.ams.org/mathscinet-getitem?mr=#1}{#2}
}
\providecommand{\href}[2]{#2}

\end{document}